\documentclass[a4paper,abstracton]{scrartcl}
\usepackage{amsfonts}
\usepackage{amssymb}
\usepackage{amsmath}
\usepackage{amsthm}
\usepackage{dsfont}
\usepackage{pgfplots}
\pgfplotsset{compat=1.3}
\usepackage{pgfplotstable}
\usepackage{endnotes}
\usepackage{graphicx}
\usepackage{multirow}
\usepackage{chngpage}
\usepackage{placeins}
\usepackage{pgfplots}
\usepackage{datatool}
\usepackage{pstricks}
\usepackage{pst-bar}
\usepackage{pst-plot}
\usepackage{tikz}
\usepackage{pgfkeys}

\usepackage{booktabs}

\usepackage[ruled,vlined,linesnumbered]{algorithm2e}

\newcommand\Eps{\frac{1}{2}}
\newcommand\GammaA{\ensuremath{\Gamma}}
\newcommand\GammaE{\ensuremath{\Gamma'}}

\newcommand\BC{\texttt{BC}\xspace}
\newcommand\WCI{\texttt{WC}$_I$\xspace}
\newcommand\WCGamma{\texttt{WC}$_\Gamma$\xspace}
\newcommand\RegretI{\texttt{Regret}$_I$\xspace}
\newcommand\RegretGamma{\texttt{Regret}$_\Gamma$\xspace}
\newcommand\RR[1]{\texttt{BR}(#1)\xspace}

\newlength{\ThreeColumnWidthAlpha}
\newlength{\ThreeColumnWidthBeta}
\newlength{\ThreeColumnWidthGamma}
\usepackage{graphicx}
\usepackage{color}

\usepackage{subfigure}

\usepackage{a4wide}

\newtheorem{theorem}{Theorem}
\newtheorem{observation}[theorem]{Observation}
\newtheorem{lemma}[theorem]{Lemma}

\usepackage{authblk}

\allowdisplaybreaks

\begin{document}

\newcommand{\X}{{\mathcal{X}}}
\newcommand{\cU}{{\mathcal{U}}}
\newcommand{\cI}{{\mathcal{I}}}
\newcommand{\cC}{{\mathcal{C}}}
\newcommand{\amax}{{a_{\max}}}
\newcommand{\amin}{{a_{\min}}}

\title{Combinatorial Optimization Problems with Balanced Regret}

\author{Marc Goerigk}
\author{Michael Hartisch\footnote{Corresponding author. Email: \texttt{michael.hartisch@uni-siegen.de}}}

\affil{Network and Data Science Management, University of Siegen,\\Unteres Schlo{\ss}~3, 57072 Siegen, Germany}

\date{}

\maketitle

\begin{abstract}
For decision making under uncertainty, min-max regret has been established as a popular methodology to find robust solutions. In this approach, we compare the performance of our solution against the best possible performance had we known the true scenario in advance. We introduce a generalization of this setting which allows us to compare against solutions that are also affected by uncertainty, which we call balanced regret. Using budgeted uncertainty sets, this allows for a wider range of possible alternatives the decision maker may choose from. We analyze this approach for general combinatorial problems, providing an iterative solution method and insights into solution properties. We then consider a type of selection problem in more detail and show that, while the classic regret setting with budgeted uncertainty sets can be solved in polynomial time, the balanced regret problem becomes NP-hard. In computational experiments using random and real-world data, we show that balanced regret solutions provide a useful trade-off for the performance in classic performance measures.
\end{abstract}



\section{Introduction}

\subsection{Motivation}

When making decisions, we usually do not have full knowledge of all aspects of the problem or consequences of the alternatives we choose from. We should therefore take uncertainty into account in the decision making process. Robust optimization \cite{gorissen2015practical,goerigk2016algorithm} is a leading paradigm to do so, encompassing different decision criterions. These include the min-max approach, where we find a decision that optimizes its worst-case performance, and the min-max regret setting, where we minimize the largest regret over all scenarios, given as the difference to the best possible objective value we could have achieved had we known the true parameters in advance. It is well-known that there is no decision approach that can fulfill a complete set of reasonable axioms of decision making simultaneously \cite{french1986decision}, i.e., from a theoretical perspective, no such approach can be superior over all others.
Hence, both min-max and min-max regret approaches have seen a wide range of research interest over the last decades, see, e.g. the surveys \cite{kouvelis2013robust,Aissi2009427,kasperski2016robust}.

Optimization problems under uncertainty are typically harder to solve than their counterparts without uncertainty, the so-called nominal problems. For most combinatorial optimization problems that can be solved in polynomial time if all parameters are known exactly, the min-max regret problem becomes NP-hard. This has led to the development of approximation methods \cite{Aissi2005}. In particular the midpoint heuristic has seen much interest, where we simply solve a nominal problem with respect to an average scenario, see \cite{Kasperski2006177,Conde2012452}, and \cite{chassein2015new,gilbert2017double,chassein2018scenario} for further developments.

Due to the computational challenge that min-max regret problems pose, algorithms have been developed for specific problems, such as knapsack \cite{furini2015heuristic}, spanning tree \cite{kasperski2012tabu}, network optimization \cite{averbakh2004interval}, assignment \cite{wu2018exact}, shortest path \cite{montemanni2004exact}, or the travelling salesman problem \cite{montemanni2007robust}.

One type of problem that has seen particular attention is the so-called selection problem \cite{kasperski2017robust,lachmann2021linear}. Depending on the variant we consider, one or multiple disjunct sets of items are given, and we want to find a subset of specific cardinality of each of these sets, such that item costs are minimized. Note that if all item costs are known, this problem can be solved in linear time \cite{cormen2009introduction}. So far, this type of problem seems to be the only case of min-max regret that remains solvable in polynomial time, see \cite{averbakh2001complexity,conde2004improved}, and also \cite{deineko2013complexity,dolgui2012min}. Selection problems arise as special cases in more complex combinatorial optimization problems, such as knapsack, assignment, scheduling, spanning tree or shortest path problems, see \cite{kasperski2013approximating}.

To define a min-max regret problem, we need to determine the set of scenarios against which we wish to protect. In most cases, this uncertainty set is either a list of scenarios (discrete uncertainty), or consists of all scenarios that adhere to lower and upper bounds for each item (interval uncertainty), see \cite{Aissi2009427}. Rarely, other uncertainty sets are considered, such as ellipsoidal uncertainty \cite{chassein2017minmax} or budgeted uncertainty \cite{poursoltani2019adjustable}, due to the resulting problem complexity. At the same time, interval uncertainty sets are by nature very conservative: a well-known result states that to calculate the regret of a fixed combinatorial solution, we need to consider the scenario where every item we want to pack is as expensive as possible, and every item we do not want to pack is as cheap as possible. This requires us to search for alternative models that still remain tractable. One such approach allows the decision maker to actively alter the uncertainty set \cite{conde2017minmax,conde2019robust}. 

In this paper we propose a different approach. We consider min-max regret problems with budgeted uncertainty as introduced by \cite{bertsimas2003robust,bertsimas2004price}, i.e., we assume that not all costs can deviate from their lower bound simultaneously; instead, the amount of deviation is controlled by a parameter. Furthermore, we introduce the concept of balanced regret: After we fix a solution to the decision problem at hand and an adversarial player chooses a scenario from the uncertainty set as well as a solution against which we need to compare, we have the opportunity to \emph{also} modify item costs in a way that items the adversary packed can become more expensive. That is, instead of comparing ourselves against a nominal solution (as is the case in classic min-max regret), we compare ourselves against a robust solution, which is also affected by uncertainty, thus levelling the playing field between decision maker and adversary.

The resulting problem, which we call ``balanced regret'', has therefore an additional stage in the decision making process. It can be considered as a special type of two-stage or adjustable robust problem (see the survey \cite{yanikouglu2019survey}). While two-stage regret problems have been considered before \cite{goerigk2020combinatorial,poursoltani2019adjustable}, they do not exhibit the specific structure of balanced regret problems as considered here. We give a formal definition of the balanced regret setting and point out our contributions in the following.

\subsection{Formal Problem Definition}

Let a nominal combinatorial optimization problem be given as
\[
\min_{\pmb{x}\in\X} \pmb{c}^t \pmb{x}\, ,
\]
where $\X\subseteq\{0,1\}^n$ denotes the set of feasible solutions, and $n$ is the problem dimension. In the following, we use the notation $[n] = \{1,2,\ldots,n\}$. Following the robust optimization paradigm, we assume that the cost vector $\pmb{c}$ is uncertain, but stemming from a known set $\cU$ of possible scenarios. The (min-max) regret problem is then to solve
\[ 
\min_{\pmb{x}\in\X} \max_{\pmb{c}\in\cU} \left( \pmb{c}^t \pmb{x} - \min_{\pmb{y}\in\X}\pmb{c}^t\pmb{y}\right) = 
\min_{\pmb{x}\in\X} \max_{\pmb{c}\in\cU,\pmb{y}\in\X} \pmb{c}^t (\pmb{x}-\pmb{y})\, ,
\]
while the min-max problem is defined as
\[ \min_{\pmb{x}\in\X} \max_{\pmb{c}\in\cU} \pmb{c}^t \pmb{x}\, . \]
In this paper, we consider budgeted uncertainty sets of the form
\[ \cU(\Gamma) = \{ \pmb{c}\in\mathbb{R}^n: c_i = \hat{c}_i + d_i \delta_i\ \forall i\in[n],\ \pmb{\delta}\in\Delta(\Gamma)\} \]
where 
\[ \Delta(\Gamma) = \{ \pmb{\delta}\in\{0,1\}^n: \sum_{i\in[n]} \delta_i \le \Gamma \}\, , \]
i.e., each item $i$ has a base cost $\hat{c}_i\geq 0$, and a deviation $d_i \geq 0$. Only up to $\Gamma$ items can deviate from their base costs simultaneously.

We introduce an extension of the min-max regret problem, where we allow the decision maker in an additional stage to increase the objective value of the adversary solution as well. That is, the balanced regret problem we consider is defined as follows:
\[ \min_{\pmb{x}\in\X}\ \max_{\pmb{\delta}\in\Delta(\Gamma), \pmb{y}\in\X}\  \min_{\pmb{\epsilon}\in\Delta(\Gamma')} \sum_{i\in[n]} (\hat{c}_i + d_i\delta_i + d_i\epsilon_i) (x_i - y_i) \]
Note that if $\Gamma'=0$, we recover a classic regret problem with budgeted uncertainty. In particular, if $\Gamma=n$ and  $\Gamma'=0$, this becomes the well-known regret problem with an interval uncertainty set. The problem we define is hence a generalization of the classic approach; this means that all hardness results extend to our setting.

As a subproblem, consider the case that a solution $\pmb{x}\in\X$ is given, and we want to calculate its objective value. We refer to this as the adversarial problem, which is given as
\[ \max_{\pmb{\delta}\in\Delta(\Gamma), \pmb{y}\in\X}\ \min_{\pmb{\epsilon}\in\Delta(\Gamma')} \sum_{i\in[n]} (\hat{c}_i + d_i\delta_i + d_i\epsilon_i) (x_i - y_i)\, .\]
If $\pmb{x}\in\X$, $\pmb{\delta}\in\Delta(\Gamma)$, and $\pmb{y}\in\X$ are fixed, what remains is called the balancing problem, given as
\[ \min_{\pmb{\epsilon}\in\Delta(\Gamma')} \sum_{i\in[n]} (\hat{c}_i + d_i\delta_i + d_i\epsilon_i) (x_i - y_i)\, . \]

There are different ways to interpret the balanced regret setting. The uncertainty affects both $\pmb{x}$ and $\pmb{y}$. Hence, both $\pmb{\delta}$ and $\pmb{\epsilon}$ are controlled by "nature", once working against $\pmb{x}$, and once against $\pmb{y}$. This means that we want to obtain a solution that minimizes the regret compared to a robust (instead of nominal) solution. Note that for fixed $\pmb{x}$ and $\pmb{\delta}$, the problem of finding the optimal $\pmb{y}$ is given by
\[ \min_{\pmb{y} \in \mathcal{X}} \left( \sum_{i \in [n]}  (\hat{c}_i + d_i \delta_i) y_i  + \max_{\pmb{\epsilon} \in \Delta (\Gamma')}  \sum_{i\in [n]} \epsilon_i\hat{d}_i y_i\right) \]
with
\[ \hat{d}_i=\begin{cases} 
d_i, & \text{ if $x_i=0$}\\
0, &\text{ if $x_i=1$}
\end{cases}\ \forall i\in[n]\, .
\]
Hence, the adversarial solution is a robust solution with respect to the updated costs.

Another way to consider this setting is that we want to find out what a best general solution $\pmb{x}$ would look like that is good for any scenario $\pmb{\delta}$. We compare it to solutions $\pmb{y}$ that are specialized to scenario $\pmb{\delta}$. However, after having chosen $\pmb{y}$ this solution must also be valid in future scenarios $\pmb{\epsilon}$.

While the balanced regret approach can be applied to any combinatorial problem, we also study the multi-representative selection problem in more detail, where $\X = \{ \pmb{x}\in\{0,1\}^n : \sum_{i\in T_\ell} x_i = p_\ell \ \forall \ell\in[L]\}$ for a partition $T_1 \cup T_2 \cup \ldots T_L = [n]$ and integers $p_\ell \le |T_\ell|$. If $L=1$, then this is also known simply as the selection problem. Selection problems play a major role in the analysis of robust discrete optimization, as they tend to lie on the boundary between NP-hard and polynomially solvable problems; the min-max regret selection problem with interval uncertainty is known to belong to the latter \cite{averbakh2001complexity}. 

To give an additional intuition for the selection problem with balanced regret, consider two persons who may choose $p$ from $n$ prizes in a lottery. While both players have an estimate of the values of these prizes, they are subjective and can be reduced when bad-mouthed by the other player. Player 1 has to choose first, and player 2 can observe this choice. Both players then reduce the values of some of the prizes so that their own choice compares as favorably as possible to the other player's choice. Which prizes should player 1 choose?

We present two numerical examples to illustrate the balanced regret approach.
As our first example, we consider the selection problem with $n=5$, $p=2$, $\Gamma=\Gamma'=1$, and costs given in Table~\ref{tab:second}.

\begin{table}[h!tb]
\centering
\caption{Item costs in the first example.\label{tab:second}}
\begin{tabular}{lrrrrr}
$i$&1&2&3&4&5\\\midrule
$\hat{c}_i$&8&5&2&17&15\\
$d_i$&9&14&15&12&1
\end{tabular}
\end{table}

An optimal balanced regret solution is to select items $1$ and $3$. An optimal adversarial decision is then to increase the cost of item $3$, while selecting items $1$ and $5$. In the balancing stage, an optimal decision is to increase the cost of item $5$. Hence, the overall costs for the decision maker are $\hat{c}_1 + \hat{c}_3 + d_3 = 8+2+15=25$ while the adversarial costs are $\hat{c}_1 + \hat{c}_5 + d_5 = 8+15+1=24$, resulting in the overall optimal balanced regret of $1$. This shows that the objective value of a balanced regret problem can be non-zero, even if $\Gamma=\Gamma'$, i.e., there is an advantage for the adversarial player by choosing her solution with the knowledge of the decision maker's solution.

As our second example, 
consider a selection problem with $n=6$ items of which $p=3$ have to be selected. We assume that $\Gamma=2$. Item costs are presented in Table~\ref{tab:one}.

\begin{table}[h!tb]
\centering
\caption{Item costs in the second example.\label{tab:one}}
\begin{tabular}{lllllll}
$i$&1&2&3&4&5&6\\\midrule
$\hat{c}_i$&3&2&1&4&4&4\\
$d_i$&2&4&4&0&0&0
\end{tabular}
\end{table}

An optimal worst-case (min-max) solution is to select the items $\{4,5,6\}$, both with respect to interval uncertainty as well as to budgeted uncertainty. An optimal solution with minimal regret with respect to $\cU(\Gamma)$ is to select items $\{1,2,3\}$. For the balanced regret problem with $\Gamma'=1$, it is optimal to select items $\{3, 4, 5\}$. In Table~\ref{tab:objone}, these three solutions are evaluated with respect to their min-max worst-case objective (WC), their regret (R) and their balanced regret (BR) value.

\begin{table}[h!tb]
\centering
\caption{Objective values in the second example.\label{tab:objone}}
\begin{tabular}{lrrrr}
 & items & WC-obj & R-obj & BR-obj \\\midrule
WC-solution & $\{4,5,6\}$ & 12 & 6 & 2\\
R-solution & $\{1,2,3\}$ & 14 & 3 & 3\\
BR-solution & $\{3,4,5\}$ & 13 & 4 & 1
\end{tabular}
\end{table}

In this example, we note that the balanced regret solution gives a trade-off between the worst-case and the classic regret criterion.

Having introduced the balanced regret approach, the remainder of this paper makes the following analysis. In Section~\ref{sec:general}, we study general properties of the problem. We show that for $\Gamma'=n$, any optimal min-max solution is also optimal for balanced regret. We give a mixed-integer programming formulation for the adversarial problem, which can be used to solve the balanced regret problem in an iterative scenario generation procedure. We also note a quantified programming formulation for the problem. In Section~\ref{sec:selection}, we focus on the multi-representative selection problem in combination with balanced regret. We prove that already the two special cases of selection (where $L=1$) and of representative selection (where $p_\ell = 1$ for all $\ell\in[L]$) are NP-hard, while the adversarial problem remains solvable in polynomial time. We further show that if either $\hat{\pmb{c}}$ or $\pmb{d}$ is a constant value, the balanced regret problem can be solved efficiently. We are also able to determine a solution with objective value zero in polynomial time, if it exists. Finally, the case where $\Gamma'=0$, i.e., the classic min-max regret criterion with budgeted uncertainty, is shown to be solvable in polynomial time. We present computational experiments using selection problems, knapsack problems, and a real-world shortest path problem in Section~\ref{sec:experiments}. Our numerical data suggest that the solutions found with the balanced regret approach give a reasonable trade-off between classic regret solutions and worst-case solutions, thus extending the pool of alternatives for a decision maker. We conclude our paper and point out further research questions in Section~\ref{sec:conclusions}.

\section{Problem Properties and Solution Methods}
\label{sec:general}

We consider the general case where $\mathcal{X} \subseteq \{0,1\}^n$ denotes the set of feasible solutions for some combinatorial optimization problem. We assume that $\X$ can be described as the intersection of a polyhedron with $\{0,1\}^n$, where the polyhedron is not necessarily integral.

\subsection{Problem Properties}

We first study basic properties of the balanced regret problem, beginning with the adversarial problem. To this end, we first consider the balancing problem for fixed $\pmb{x}$, $\pmb{\delta}$ and $\pmb{y}$, which is given as 
\begin{subequations}
\label{eq:revenge}
\begin{align}
\min\ & \sum_{i\in[n]} (\hat{c}_i + d_i \delta_i + d_i \epsilon_i) (x_i - y_i) \\
\text{s.t. } & \sum_{i\in[n]} \epsilon_i \le \Gamma' \\
& \epsilon_i \in\{0,1\} & \forall i\in[n] 
\end{align}
\end{subequations}
We first note the following.
\begin{lemma} \label{lemma:NoEpsilonIfX}
There is an optimal solution to the balancing problem where $\epsilon_i + x_i \le 1$, i.e., no item is attacked by the decision maker that is also packed by the decision maker.
\end{lemma}
\begin{proof}
Let an optimal solution to problem~\eqref{eq:revenge} be given where $x_i=\epsilon_i=1$ for some $i\in[n]$. Setting $\epsilon_i=0$ then results in a new solution that is feasible as well. Because $d_i(x_i-y_i) \ge 0$, the new solution has an objective value that is not less than the old solution and is hence optimal as well.
\end{proof}
Hence, an optimal solution to this problem is to sort items by cost coefficients $d_i(x_i-y_i)$ and to pack up to $\Gamma'$ many items with smallest non-positive costs. This means we can consider the linear relaxation of problem~\eqref{eq:revenge} to find an optimal integral solution. By dualizing the balancing problem, a compact non-linear formulation of the adversarial problem can be derived as follows.
\begin{subequations}
\label{eq:revenge2}
\begin{align}
\max\ & \sum_{i\in[n]} (\hat{c}_i + d_i\delta_i)(x_i - y_i) - \Gamma' s - \sum_{i\in[n]} t_i \\
\text{s.t. } &  s + t_i \ge d_i (y_i - x_i) & \forall i\in[n] \\
& \sum_{i\in[n]} \delta_i \le \Gamma \\
& \pmb{y}\in \X \\
& \delta_i \in\{0,1\} & \forall i\in[n] \\
& s \ge 0 \\
& t_i \ge 0 & \forall i\in[n]
\end{align}
\end{subequations}

\begin{lemma}\label{lemma:NoDeltaIfY}
There is an optimal solution to the adversarial problem where $y_i + \delta_i \le 1$, i.e., no item is attacked by the adversary that is also packed by the adversary.
\end{lemma}
\begin{proof}
Let an optimal solution to problem~\eqref{eq:revenge2} be given where $y_i=\delta_i=1$ for some $i\in[n]$. Setting $\delta_i=0$ then results in a new solution that is feasible as well. Because $d_i(x_i-y_i) \le 0$, the new solution has an objective value that is not less than the old solution and is hence optimal as well.
\end{proof}

Using this insight, we formulate the following linearized problem version of the adversarial problem.
\begin{subequations}
\label{eq:revenge3}
\begin{align}
\max\ & \sum_{i\in[n]} (\hat{c}_i + d_i\delta_i) x_i - \sum_{i\in[n]} \hat{c}_i y_i - \Gamma' s - \sum_{i\in[n]} t_i \\
\text{s.t. } &  s + t_i \ge d_i (y_i - x_i) & \forall i\in[n] \\
& y_i + \delta_i \le 1 & \forall i\in[n] \\
& \sum_{i\in[n]} \delta_i \le \Gamma \\
& \pmb{y}\in\X \\
& \delta_i \in\{0,1\} & \forall i\in[n] \\
& s \ge 0 \\
& t_i \ge 0 & \forall i\in[n]
\end{align}
\end{subequations}

As we can formulate the adversarial problem as a mixed-integer linear program, we can make the following observation.

\begin{observation}\label{th:advnp}
The adversarial problem is in NP.
\end{observation}

Recall that the min-max regret knapsack problem with interval uncertainty is $\Sigma^p_2$-hard \cite{deineko2010pinpointing}. It follows from Observation~\ref{th:advnp} that the balanced regret version remains in $\Sigma^p_2$. This means that in this case, the complexity class does not increase. As we see in Section~\ref{sec:selregrev}, the complexity does increase for multi-representative selection problems.

In the following, we use the notation $[x]_+$ to denote $\max\{0,x\}$.
\begin{lemma}\label{lemma:revenge1}
There is an optimal solution to problem~\eqref{eq:revenge3} with $t_i = [d_i(y_i-x_i)-s]_+$ and $s\in \mathcal{S}=\{0\}\cup\{d_i : i\in[n]\}$.
\end{lemma}
\begin{proof}
The proof idea is the same as the classic result from \cite{bertsimas2003robust}. For fixed $(\pmb{y},\pmb{\delta},s)$, we would like to choose values $t_i$ that are as small as possible, i.e., $t_i = \max\{d_i(y_i-x_i)-s,0\}$. Substituting for variables $t_i$, the resulting optimization problem for fixed $(\pmb{y},\pmb{\delta})$ is piece-wise linear with break points within the set $\{d_i : i\in[n]\}$. Hence, the claim follows.
\end{proof}
Using Lemma~\ref{lemma:revenge1}, we can enumerate possible choices for variable $s\in \mathcal{S}$. Thus, the adversarial problem is equivalent to:
\begin{subequations}
\label{eq:revenge4}
\begin{align}
\max_{s\in \mathcal{S}} \Big\{\max\ & \sum_{i\in[n]} (\hat{c}_i + d_i\delta_i) x_i - \sum_{i\in[n]} \hat{c}_i y_i - \Gamma' s - \sum_{i\in[n]} [d_i - d_i x_i - s]_+y_i \\
\text{s.t. } 
& y_i + \delta_i \le 1 & \forall i\in[n] \label{p4-1}\\
& \sum_{i\in[n]} \delta_i \le \Gamma \label{p4-2}\\
& \pmb{y}\in\X \label{p4-3}\\
& \delta_i \in\{0,1\} & \forall i\in[n] \Big\}
\end{align}
\end{subequations}
Note that even if the polyhedron of $\X$ can be described by a totally unimodular coefficient matrix, problem~\eqref{eq:revenge4} is not necessarily totally unimodular as well. As an example, consider the totally unimodular matrix
\[ B=\begin{pmatrix} 1 & 1 & 0 \\ 1 & 0 & 1 \end{pmatrix} \]
or any matrix containing $B$ as a submatrix. Then the resulting coefficient matrix of problem~\eqref{eq:revenge4} is not totally unimodular.

We further note the following general results.

\begin{lemma}\label{lemma:nonneg}
The adversarial problem always has non-negative optimal objective value.
\end{lemma}
\begin{proof}
Let any $\pmb{x}\in\X$ be given. Set $\pmb{y}=\pmb{x}$. Due to Lemmas~\ref{lemma:NoEpsilonIfX} and \ref{lemma:NoDeltaIfY}, no item is attacked, and the corresponding adversarial objective value is 0. Hence, the optimal value is at least as large.
\end{proof}

\begin{lemma}\label{lemma:increase}
For a fixed solution $\pmb{x}\in\X$, the objective value of the adversarial problem is non-decreasing with respect to $\Gamma$.
\end{lemma}
\begin{proof}
The claim follows directly from the fact that a higher value of $\Gamma$ results in a larger set of feasible solutions in Problem~\eqref{eq:revenge3}.
\end{proof}

We can now show that the case $\Gamma'=n$ can be solved by solving a single nominal problem, independently of the choice of $\Gamma$.

\begin{theorem}
Any optimal solution for the nominal problem with costs $\hat{\pmb{c}}+\pmb{d}$ is also optimal for the balanced regret problem with $\Gamma'=n$.
\end{theorem}
\begin{proof}
Consider first the case $\Gamma=\Gamma'=n$. Then the balanced regret problem becomes
\[ \min_{\pmb{x}\in\X} \max_{\pmb{y}\in\X} \left( (\hat{\pmb{c}}+\pmb{d})^t\pmb{x} - (\hat{\pmb{c}}+\pmb{d})^t\pmb{y} \right) = \left(\min_{\pmb{x}\in\X} (\hat{\pmb{c}}+\pmb{d})^t\pmb{x}\right) -  \left(\min_{\pmb{y}\in\X} (\hat{\pmb{c}}+\pmb{d})^t\pmb{y}\right) = 0\, , \]
i.e., both $\pmb{x}$ and $\pmb{y}$ solutions need to plan for expensive item costs.
Note that the problems in $\pmb{x}$ and $\pmb{y}$ have become independent; thus, the claim holds for this case.
For $\Gamma=n$, the objective value of any nominal minimizer of $\hat{\pmb{c}}+\pmb{d}$ is already zero. Hence, these solutions remain optimal for smaller values of $\Gamma$, due to Lemma~\ref{lemma:nonneg} and Lemma~\ref{lemma:increase}.
\end{proof}

For multi-stage problems there can be a difference in complexity depending on whether continuous or budgeted uncertainty is used. For example, the two-stage selection problem with continuous budgeted uncertainty can be solved in polynomial time \cite{chassein2018recoverable}, but becomes NP-hard for discrete budgeted uncertainty \cite{goerigk2020recoverable}. Observe that this is not the case for balanced regret. Let any $\pmb{y}\in\X$ be fixed in the adversarial problem~\eqref{eq:revenge2}. Then the $\pmb{\delta}$ variables can be relaxed, as there is an optimal integral solution for the relaxed problem. Hence, no differentiation between continuous and discrete budgeted uncertainty is necessary.

\subsection{Solution Methods\label{SubSec::SolutionMethods}}

Recall that the balanced regret problem is given as
\begin{equation}\label{regrev}
\min_{\pmb{x}\in\X}\ \max_{\pmb{\delta}\in\Delta(\Gamma), \pmb{y}\in\X}\ \min_{\pmb{\epsilon}\in\Delta(\Gamma')} \sum_{i\in[n]} (\hat{c}_i + d_i\delta_i + d_i\epsilon_i) (x_i - y_i)\, .
\end{equation}
We write 
\[ \Xi = \X \times \Delta(\Gamma)\, . \]
Note that $\Xi$ is a finite set (albeit of exponential size), which means we can enumerate its elements $\Xi = \{ (\pmb{y}^1,\pmb{\delta})^1, \ldots, (\pmb{y}^K,\pmb{\delta}^K)\}$ with $K=|\Xi|$. Hence, problem~\eqref{regrev} is equivalent to the following problem.
\begin{subequations}
\label{regrev2}
\begin{align}
\min\ & z \\
\text{s.t. } & z \ge \sum_{i\in[n]} (\hat{c}_i + d_i\delta^k_i + d_i \epsilon^k_i)(x_i-y^k_i) & \forall k\in[K] \label{eqnonlinear} \\
& \sum_{i\in[n]} \epsilon^k_i \le \Gamma'  & \forall k\in[K]\\
& \pmb{\epsilon}^k \in \{0,1\}^n & \forall k\in[K]\\
& \pmb{x}\in\X 
\end{align}
\end{subequations}
Note that the product $\epsilon^k_ix_i$ is non-linear. 
It can be linearized by adding constraints $\epsilon^k_i + x_i \le 1$ for all $i\in[n]$, $k\in[K]$ and changing constraints~\eqref{eqnonlinear} to
\[ z \ge \sum_{i\in[n]} (\hat{c}_i + d_i\delta^k_i)x_i - (\hat{c}_i + d_i \delta^k_i + d_i \epsilon^k_i)y^k_i \quad \forall k\in[K] \, .\]
Solving this problem with all scenarios is denoted as the enumeration approach.
To avoid the full enumeration of $\Xi$, the following iterative method can be used (see also \cite{zeng2013solving}). For any subset $\Xi'\subseteq\Xi$, solving problem~\eqref{regrev2} gives a lower bound. The true objective value of the resulting solution $\pmb{x}$ can be evaluated by solving the adversarial problem. While the resulting upper bound and the current lower bound do not coincide, we add the solution to the adversarial problem to the current set of scenarios $\Xi'$ and repeat the process. As $\Xi$ is finite, this method ends after a finite number of iterations. We refer to this approach as the iterative solution method.

Another option is to formulate the balanced regret problem as a quantified program, which can then be solved using a general open-source solver, such as Yasol \cite{YasolACG17,DissMichael}. This solver has the advantage of being able to solve general multistage robust discrete linear optimization problems with polyhedral and decision-dependent uncertainty sets, 
without the need to reformulate them into mixed-integer programs.
In \cite{goerigk2021multistage}, Yasol was already used to solve multi-stage robust optimization problems.

As before, we write $\Xi = \X \times \Delta(\Gamma)$,
i.e. the universally quantified variables must provide a valid solution to the basic problem ($\pmb{y} \in \mathcal{X}$)  and specify at most $\Gamma$ indices for which the objective value increases ($\pmb{\delta}\in \Delta(\Gamma)$). The corresponding quantified program is then as follows:
\begin{align*}
\min\ & \sum_{i\in[n]} \hat{c}_i (x_i - y_i) + \sum_{i\in[n]} d_i (\delta_i+\epsilon_i)(x_i -y_i)  \\
\textnormal{s.t.}\ &\exists \pmb{x} \in \{0,1\}^n \quad \forall (\pmb{y},\pmb{\delta})\in \Xi \quad \exists \pmb{\epsilon}\in \{0,1\}^n : \nonumber \\ 
&\sum_{i \in [n]} \epsilon_i \leq \Gamma'  \\
&\pmb{x} \in \mathcal{X}
\end{align*}
This formulation has to be linearized for Yasol. To this end, new variables $\alpha^x_i$ and $\alpha^y_i$ are introduced in order to represent a cost increase of item $i$ selected via $\pmb{x}$ and $\pmb{y}$, respectively. As those variables only process information given by the variables $\pmb{x}$, $\pmb{\delta}$, $\pmb{y}$ and $\pmb{\epsilon}$, they are placed in the final existential variable block.
\begin{align*}
\min\ & \sum_{i\in[n]} \hat{c}_i x_i + \sum_{i\in[n]} d_i \alpha^x_i - \sum_{i\in[n]}\hat{c}_i y_i - \sum_{i\in[n]} d_i \alpha^y_i  \\
\textnormal{s.t.}\ &\exists \pmb{x} \in \{0,1\}^n \quad \forall (\pmb{y},\pmb{\delta})\in \Xi   \quad \exists \pmb{\epsilon}\in \{0,1\}^n,\ \pmb{\alpha}^x\in \{0,1\}^n,\  \pmb{\alpha}^y\in \{0,1\}^n : \nonumber \\ 
&\sum_{i \in [n]} \epsilon_i \leq \Gamma  \\
&\pmb{\alpha}^x \geq (\pmb{\delta}+\pmb{\epsilon})+ \pmb{x}-\pmb{1}\\
&\pmb{\alpha}^y \leq  (\pmb{\delta}+\pmb{\epsilon})\\
&\pmb{\alpha}^y \leq \pmb{y}\\
&\pmb{x} \in \mathcal{X}\\
\end{align*}
In Section~\ref{sec:experiments}, we compare the performance when solving balanced regret problems with the iterative solution method and with Yasol.

\section{Balanced Regret for Multi-Representative Selection}
\label{sec:selection}

\subsection{Problem Hardness}
\label{sec:selregrev}

We now consider the multi-representative selection problem (see, e.g., \cite{goerigk2020recoverable}), where $\X = \{ \pmb{x}\in\{0,1\}^n : \sum_{i\in T_\ell} x_i = p_\ell \ \forall \ell\in[L]\}$ for a partition $T_1 \cup T_2 \cup \ldots \cup T_L = [n]$ and integers $p_\ell \le |T_\ell|$. Recall that this contains the selection problem (where $L=1$) and the representative selection problem (where $p_\ell=1$ for all $\ell\in[L]$) as special cases. We show that the balanced regret problem is NP-hard for both special cases.

\begin{theorem}
Balanced regret selection (i.e., when $L=1$) with budgeted uncertainty is NP-hard.
\end{theorem}
\begin{proof}
We use the weakly NP-hard equipartition problem (see \cite{garey1979computers}) with an even number $n \in \mathbb{N}$ of items and weights $a_i  \in \mathbb{N}$, $i \in [n]$, with the problem statement: is there a subset $S \subseteq [n]$, $|S|=\frac{n}{2}$, with $\sum_{i \in S}a_i = \sum_{i \in [n]\setminus S} a_i$? Let $A=\sum_{i \in [n]}a_i$.
We construct a selection problem with balanced regret by setting the parameters $\hat{\pmb{c}}$ and $\pmb{d}$ as given in Table~\ref{Tab::Reduction}, dividing them into the three sets $\alpha$, $\beta$ and $\gamma$.

\begin{table}[htb]
\begin{center}
\caption{Items costs for the balanced regret selection problem.\label{Tab::Reduction}}
\begin{tabular}{l|lll|lll|ll}
\multicolumn{1}{c}{}&\multicolumn{3}{c}{$\alpha$ ($n$ items)}
&\multicolumn{3}{c}{$\beta$ ($2n+2$ items)}&\multicolumn{2}{c}{$\gamma$ ($2$ items)}\\
\multicolumn{1}{c}{}&\multicolumn{3}{c}{$\overbrace{\hspace{\ThreeColumnWidthAlpha}}$}&\multicolumn{3}{c}{$\overbrace{\hspace{\ThreeColumnWidthBeta}}$}&\multicolumn{2}{c}{$\overbrace{\hspace{\ThreeColumnWidthGamma}}$}\\
$i$&$1$&$\ldots$&$n$&$n+1$&$\ldots$ &$3n+2$&$3n+3$&$3n+4$
\\\midrule
$\hat{c}_i$&$a_1$&$\ldots$&$a_n$ &$0$&$\ldots$&$0$&$0$&$0$\\
$d_i$&$A-\frac{3}{2}a_1$&$\ldots$&$A-\frac{3}{2}a_n$&$\frac{3}{2}A-\frac{1}{4}$&$\ldots$&$\frac{3}{2}A-\frac{1}{4}$&$A-\frac{1}{4}$&$A-\frac{1}{4}$
\end{tabular}
\end{center}
\end{table}

We need to select $\frac{n}{2}+1$ items of the $3n+4$ items that are available. We further set
$\GammaA=\frac{n}{2}+1$ and $\GammaE=1$ to complete the description of the instance.

Note that for the first-stage decision, we can assume that no item from $\beta$ is selected. Otherwise, it can be exchanged for an item from $\alpha$ or $\gamma$, giving non-increasing costs for the first-stage solution, and non-decreasing costs for the adversary solution.
Hence, three cases remain, varying in the number of items selected from $\gamma$. In each case we denote by $S\subseteq [n]$ the set of items selected from $\alpha$ and write $X=\sum_{i \in S} a_i$. For the adversary decision, note that if one of the items from $\beta$ is selected, then selecting all $\frac{n}{2}+1$ items from $\beta$ is reasonable, as $\GammaE=1$. Also note that the adversary will always increase the costs of all items selected in the first decision stage.
\begin{enumerate}
\item \textit{First-stage decision selects no item from $\gamma$ and $\frac{n}{2}+1$ items from $\alpha$.}\\
The first stage cost (including the cost increase) are $(\frac{n}{2}+1)A-\frac{X}{2}$. The adversary then has two choices:
\begin{enumerate}
\item \textit{Adversary selects all items from $\beta$.} In the third (balancing) stage, the cost of one of those items is increased, resulting in the overall objective value 
$(\frac{n}{2}+1)A-\frac{X}{2} - (\frac{3}{2}A - \frac{1}{4}) = \frac{2n-2}{4}A-\frac{X}{2}+\frac{1}{4}$.

\item \textit{Adversary selects $\frac{n}{2}-1$ items from $\alpha$ not selected in the first stage and both from $\gamma$.}
The cost of one of the items from $\gamma$ is increased, resulting in the overall objective value
$ (\frac{n}{2}+1)A-\frac{X}{2} - (A-X+A - \frac{1}{4}) = \frac{2n-4}{4}A+\frac{X}{2}+ \frac{1}{4}$.
\end{enumerate}
We conclude that the worst-case objective value of this case is
\begin{align*}
& \max \left\{ \frac{2n-2}{4}A-\frac{X}{2}+\frac{1}{4}, \frac{2n-4}{4}A+\frac{X}{2}+ \frac{1}{4}\right\}  \\
= &\frac{2n-4}{4}A + \frac{1}{4} + \frac{1}{2}\max\left\{A - X, X \right\}\\
\ge & \frac{2n-3}{4}A+\frac{1}{4}\,.
\end{align*}

\item \textit{First-stage decision selects one item from $\gamma$ and $\frac{n}{2}$ items from $\alpha$.}\\
The first stage cost (including the cost increase) are $(\frac{n}{2}+1)A-\frac{X}{2}-\frac{1}{4}$. The adversary then has two choices:
\begin{enumerate}
\item \textit{Adversary selects all items from $\beta$,} resulting in the overall objective value
$(\frac{n}{2}+1)A-\frac{X}{2}-\frac{1}{4} - (\frac{3}{2}A - \frac{1}{4}) = \frac{2n-2}{4}A-\frac{X}{2}$.

\item \textit{Adversary selects $\frac{n}{2}$ items from $\alpha$ and the item from $\gamma$ for which the costs where not increased yet.}\
Then the cost of the item from  $\gamma$ is increased, resulting in the overall objective value 
$(\frac{n}{2}+1)A-\frac{X}{2}-\frac{1}{4} - (A - X + A - \frac{1}{4}) = \frac{2n-4}{4}A+\frac{X}{2}$.
\end{enumerate}
We conclude that the worst-case objective value of this case is
\[ \max\left\{ \frac{2n-2}{4}A-\frac{X}{2}, \frac{2n-4}{4}A+\frac{X}{2} \right\} = \frac{2n-4}{4}A + \frac{1}{2}\max\left\{ A - X, X \right\}\,,\]
which is minimized with value $\frac{2n-3}{4}A$, if $X=\frac{A}{2}$ can be achieved.

\item \textit{First-stage decision selects two items from $\gamma$ and $\frac{n}{2}-1$ items from $\alpha$.}\\
The first stage costs (including the cost increase) are $(\frac{n}{2}+1)A-\frac{X}{2}-\frac{1}{2}$. The adversary then has two choices:
\begin{enumerate}
\item \textit{Adversary selects all items from $\beta$,} resulting in the overall objective value
$(\frac{n}{2}+1)A-\frac{X}{2}-\frac{1}{2} - (\frac{3}{2}A - \frac{1}{4}) = \frac{2n-2}{4}A-\frac{X}{2}-\frac{1}{4}$.
\item \textit{Adversary selects $\frac{n}{2}+1$ items from $\alpha$.} Then we increase the costs of the item for which $d_i$ is maximal, i.e., $a_i$ is minimal. Let $\tilde{a}=\min_{i \in [n]\setminus S } \{a_i\}$ denote this value. This results in the overall objective value 
$(\frac{n}{2}+1)A-\frac{X}{2}-\frac{1}{2} - (A-X+A-\frac{3}{2}\tilde{a}) = 
\frac{2n-4}{4}A+\frac{X}{2}+\frac{3}{2}\tilde{a}-\frac{1}{2}$.
\end{enumerate}
We conclude that the worst-case costs of this case are
\begin{align*}
&\max\left\{  \frac{2n-2}{4}A-\frac{X}{2}-\frac{1}{4}, \frac{2n-4}{4}A+\frac{X}{2}+\frac{3}{2}\tilde{a}-\frac{1}{2} \right\}\\
= &\frac{2n-4}{4}A - \frac{1}{4} + \frac{1}{2}\max\left\{ A - X, X + 3\tilde{a} - \frac{1}{2} \right\}\\
\ge & \frac{2n-4}{4}A - \frac{1}{4} + \frac{1}{2} \left(\frac{1}{2}A + \frac{6\tilde{a}-1}{4} \right) \\
= & \frac{2n-3}{4}A + \frac{6\tilde{a}-3}{8}\, .
\end{align*}
\end{enumerate}

Therefore, if an equipartition of the weights $a_i$ exists, the optimal objective value of the built selection problem with balanced regret is $\frac{2n-3}{4}A$ and the items taken from $\alpha$ indicate the partition (case~2). Otherwise, the optimal objective value is strictly larger than $\frac{2n-3}{4}A$.
\end{proof}

\begin{theorem}
Balanced regret representative selection (i.e., when $p_\ell=1$ for all $\ell \in[L]$) with budgeted uncertainty is NP-hard.
\end{theorem}
\begin{proof}
We use the weakly NP-hard partition problem (see \cite{garey1979computers}) with $ 
n \in \mathbb{N}$ items and weights $a_i  \in \mathbb{N}$, $i \in [n]$. The question is whether there exists a subset $S\subseteq[n]$ with $\sum_{i\in S} a_i = \sum_{i\in[n]\setminus S} a_i$.
Let $A=\sum_{i \in [n]}a_i$ and  $\amax=\max_{i \in [n]}a_i$. Without loss of generality, we assume that $\amax \leq \frac{A}{3}$. Note that if this is not the case, we can modify the instance by adding two items of size $A$.

We construct a representative selection problem with balanced regret with $n$ partitions and four items per partition, by setting the parameters $\hat{\pmb{c}}$ and $\pmb{d}$ as given in Table~\ref{Tab::ReductionRepresentative} for each partition $i$. We set $\Gamma=n$ and $\Gamma'=1$. 
\begin{table}[htb]
\centering
\caption{Item costs for the balanced regret representative selection problem.\label{Tab::ReductionRepresentative}}
\begin{tabular}{l|llll}
$i$&$\alpha_i$&$\beta_i^1$&$\beta_i^2$&$\gamma_i$
\\\midrule
$\hat{c}_i$&$A+2a_i$&$0$&$0$&$A-2a_i$  \\
$d_i$&$2A-3a_i$&$(n+2)A+3\amax$&$(n+2)A+3\amax$&$2A+3a_i$
\end{tabular}

\end{table}
Since all $d_i$ values are nonnegative and $\Gamma=n$, the adversary always increases the costs of all items selected in the first stage. In the first stage, none of the $\beta$ items will be selected, as their resulting (increased) cost is much larger than the costs of items in $\alpha$ and in $\gamma$; and, as there are two of them, it will have no effect on the options of the adversary. Hence, let $S$ and $\bar{S}$ be the sets containing the indices of selected $\alpha$ and $\gamma$ items, respectively. Let $X=\sum_{i \in S}a_i$ and $\bar{X}=\sum_{i \in \bar{S}}a_i$. We distinguish three cases, varying in the number of $\alpha$ and $\gamma$ items selected.

\begin{enumerate}
\item \textit{First-stage decision selects at least one $\alpha$ and at least one $\gamma$ item.}\\
The cost of each selected representative item will be increased by the adversary. Hence, the first-stage costs are $\bar{X}-X+3nA$. For the solution selected by the adversary, two cases remain:

\begin{enumerate}
\item \textit{Adversary selects at least one $\beta$ item.}\\
If for at least one partition a $\beta$ item is selected, it is reasonable to select a $\beta$ item for each partition, as the cost of only one item can be increased in the balancing stage. Hence, the resulting overall objective value is 
$\bar{X} - X + 3nA - ((n+2)A + 3\amax) = \bar{X}-X +(2n-2)A-3\amax$.

\item \textit{Adversary selects only $\alpha$ and $\gamma$ items.}\\
Let $J\subseteq[n]$ denote the set of partitions where the adversary chooses the same item as the first-stage solution. If $J=[n]$, the adversary chooses all items as in the first-stage decision, and the objective value is zero. Note that in an optimal adversary decision, $J$ cannot contain a $\gamma$ item, as choosing the $\alpha$ item in this case always results in a better objective value for the adversary. Hence, we consider any solution with $\bar{J}=[n]\setminus J \neq \emptyset$ and with $\bar{J}\neq[n]$. We show that such solutions can be improved by adding another item to $\bar{J}$. We write $\bar{J}^\alpha\cup \bar{J}^\gamma = \bar{J}$ to denote indices where an $\alpha$ item or a $\gamma$ item is chosen by the adversary, respectively. 

If $\bar{J}^\gamma = \emptyset$, the current total costs are
\[ \sum_{i\in \bar{J}^\alpha} (3A + a_i - (A+2a_i) ) - \max_{i\in \bar{J}^\alpha} (2A-3a_i) = 
 \sum_{i\in \bar{J}^\alpha} (2A - a_i) - \max_{i\in \bar{J}^\alpha} (2A-3a_i) >0\, .
\]
On the one hand, adding an item $i$ to $\bar{J}^\alpha$ results in an increase of the objective value by $2A-a_i + 3[\min_{j \in \bar{J}^\alpha \setminus \{i\}}a_j-a_i]_+>0$. If, on the other hand, an item $i$ is added to the up to now empty set $\bar{J}^\gamma$, the objective value changes by 
$$3A-a_i -(3A+a_i)+ \max_{j\in \bar{J}^\alpha} (2A-3a_j) \geq 2A-5\amax > 0\, .$$

Now assume $\bar{J}^\gamma \neq \emptyset$. Then the total costs are
\begin{align*}
& \sum_{i\in \bar{J}^\alpha} (3A + a_i - (A+2a_i) ) + \sum_{i\in\bar{J}^\gamma} (3A - a_i - (A-2a_i)) - \max_{i\in\bar{J}^\gamma} (2A+3a_i) \\
= & \sum_{i\in \bar{J}^\alpha} (2A - a_i) + \sum_{i\in\bar{J}^\gamma} (2A + a_i) - \max_{i\in\bar{J}^\gamma} (2A+3a_i) \geq -2\amax \, .
\end{align*}
On the one hand, adding an item $i$ to $\bar{J}^\alpha$ results in an increase of the objective value by $2A-a_i >0$ resulting in a positive objective value. If, on the other hand, an item $i$ is added to  $\bar{J}^\gamma$, the objective value changes by 
$$2A+a_i + 3[a_i-\max_{j \in \bar{J}^\gamma \setminus \{i\}}a_j]_+>0 \, .$$

Hence, any solution with $\bar{J}\neq \emptyset$ and $J \neq \emptyset$  can be improved by shifting all items into $\bar{J}^\alpha$ or $\bar{J}^\gamma$.
Thus, in an optimal solution, we have $J=\emptyset$, i.e., the adversary always chooses the $\alpha$ or $\gamma$ item that the first stage did not choose.

With this observation, we can give the total costs in this case as $X-\bar{X} + (2n-2)A - 3\amax(S)$, where $\amax(S) = \max_{i\in S} a_i$ (recall that $S$ denotes the indices of partitions where the first-stage decision packs the $\alpha$ item).

\end{enumerate}

We conclude that the total costs between options a) and b) become
$$(2n-2)A + \max\{X-\bar{X}-3\amax(S),\bar{X}-X-3\amax\}\, .$$
Note that
\begin{align*}
 &\max\{X-\bar{X}-3\amax(S),\bar{X}-X-3\amax\} \\
 \ge & \max\{X-\bar{X}-3\amax,\bar{X}-X-3\amax\}\\
 = &|X-\bar{X}| - 3\amax\, .
 \end{align*}
Hence, the smallest possible objective value $(2n-2)A - 3\amax$ can only be achieved if we choose $S$ such that $\sum_{i\in S} a_i = \sum_{i\in\bar{S}} a_i$ (in this case, w.l.o.g. choose $S$ such that $\amax(S) = \amax$).

\item \textit{First-stage decision selects no $\gamma$ item and all $\alpha$ items, i.e. $S=[n]$.}\\
The first stage costs are $(3n-1)A$. It is sufficient to only consider that the adversary selects all $\gamma$ items. In this case the overall objective value amounts to $(3n-1)A -(nA-2A+2A+3\amax)  = (2n-1)A-3\amax$, which is higher than in case 1.

\item \textit{First-stage decision selects no $\alpha$ item and all $\gamma$ items, i.e. $\bar{S}=[n]$.}\\
The first stage costs are $(3n+1)A$. It is sufficient to only consider that the adversary selects all $\beta$ items, resulting in the overall objective value $(3n+1)A-(n+2)A-3\amax = (2n-1)A-3\amax$, which is higher than in case 1.
\end{enumerate}

We conclude that if a partition of the weights $a_i$ exists, the optimal value of case 1 is $(2n-2)A-3\amax$, which is smaller than the value resulting from cases 2 and 3. If no partition exists, then in each case the optimal value is strictly larger than $(2n-2)A-3\amax$.
\end{proof}

\subsection{Mixed-Integer Program and Special Cases}

We now consider the adversarial problem in the case of multi-representative selection.

\begin{lemma}\label{lemma:two}
The coefficient matrix of the adversarial problem~\eqref{eq:revenge4} for fixed $s\in\mathcal{S}=\{0\}\cup \{d_i : i\in[n]\}$ is totally unimodular for the multi-representative selection problem.
\end{lemma}
\begin{proof}
Using the Ghouila-Houri criterion, let any set of rows of the coefficient matrix be given. Assigning the value $+1$ to rows corresponding to constraints~\eqref{p4-3} (i.e., $\sum_{i\in T_\ell} y_i = p_\ell$ for all $\ell\in[L]$) and to row~\eqref{p4-2} (i.e., $\sum_{i\in[n]} \delta_i \le \Gamma$) while assigning the value $-1$ to any row of constraints~\eqref{p4-1} (i.e., $y_i + \delta_i \le 1$)
then results in a sum in $\{-1,0,1\}$ for each column.
\end{proof}

Let $\{s^1,\ldots,s^R\} = \mathcal{S}$ denote the values in set $\mathcal{S}$. Dualizing the inner maximization problem for each $r \in [R]$ then gives the following compact problem formulation for the balanced regret problem.
\begin{subequations}
\label{eq:revenge5}
\begin{align}
\min\ & t \\
\text{s.t. } & t \ge \sum_{i\in[n]} \hat{c}_i x_i + \Gamma\pi^r + \sum_{i\in[n]} \rho^r_i - \Gamma's^r - \sum_{\ell\in[L]} p_\ell\kappa^r_\ell & \forall r\in[R] \\
& \sum_{i\in T_\ell} x_i = p_\ell & \forall \ell \in[L] \\
& \pi^r + \rho^r_i \ge d_i x_i & \forall i\in[n], r\in[R] \\
& \rho^r_i + \hat{c}_i + [d_i - d_ix_i - s^r]_+ \ge \kappa^r_\ell & \forall \ell\in L, i \in T_\ell, r\in[R] \\
& x_i \in \{0,1\} & \forall i\in[n] \\
& \pi^r \ge 0 & \forall r\in[R] \\
& \rho^r_i \ge 0 & \forall i\in[n], r\in[R] 
\end{align}
\end{subequations}
Note that $[d_i - d_ix_i -s^r]_+$ can be linearized to $[d_i - s^r]_+ (1-x_i)$, as $x_i$ are binary variables.

Using this problem formulation, we show that the following dominance property holds true.
\begin{lemma}\label{lemma:dominance}
Consider two items $i,j\in T_\ell$ with $\hat{c}_i \le \hat{c}_j$ and $\hat{c}_i + d_i \le \hat{c}_j + d_j$. Then there is an optimal solution $\pmb{x}$ with $x_i \ge x_j$.
\end{lemma}
\begin{proof}
Consider some solution $\pmb{x}$ where $x_i = 0$ and $x_j=1$. We construct a new solution $\pmb{x}'$ where $x'_i=1$ and $x'_j=0$ and show that its objective value does not increase. We focus on any $r\in[R]$ and drop the symbol $r$ for ease of presentation.
The modified solution uses the same values for $\pi$ and $\kappa$ variables. For $\pmb{\rho}$, we only change the values corresponding to items $i$ and $j$, i.e., we consider the following:
\begin{align*}
\rho_i &= \max\{\kappa - \hat{c}_i - [d_i - s]_+, 0\} & (x_i &= 0) \\
\rho_j &= \max\{d_j - \pi, \kappa-\hat{c}_j,0\}  & (x_j &= 1)\\
\rho'_i &= \max\{d_i - \pi, \kappa-\hat{c}_i, 0\}  & (x'_i &= 1)\\
\rho'_j & = \max\{\kappa - \hat{c}_j - [d_j - s]_+,0\} & (x'_i &= 0)
\end{align*}
We show that $\hat{c}_j + \rho_i + \rho_j \ge \hat{c}_i + \rho'_i + \rho'_j$. Note that we have that
\[ \hat{c}_j + \rho_j = \max\{ \hat{c}_j + d_j - \pi, \kappa, \hat{c}_j\} \ge \max\{ \hat{c}_i + d_i- \pi, \kappa, \hat{c}_i\} = \hat{c}_i + \rho'_i\, . \]
\begin{enumerate}
\item  Let us first assume that $d_i \le d_j$. Then, $[d_i - s]_+ \le [d_j - s]_+$. We conclude that
\[ \rho_i = \max\{\kappa - \hat{c}_i - [d_i - s]_+, 0\} \ge  \max\{\kappa - \hat{c}_j - [d_j - s]_+,0\} = \rho'_j \]
which proves the claim in this case.

\item Now consider the case that $d_j \le d_i$. We further distinguish the following cases.
\begin{enumerate}
\item $s \le d_j \le d_i$: It holds that
\[ \rho_i = \max\{\kappa - \hat{c}_i - d_i + s,0 \} \ge \max\{\kappa - \hat{c}_j - d_j +s, 0\} = \rho'_j\, .\]
\item $d_j \le s \le d_i$: We have that $\hat{c}_j + s \ge \hat{c}_i + d_i$. Hence, $\kappa - \hat{c}_i - d_i + s \ge \kappa - \hat{c}_j$ and therefore $\rho_i \ge \rho'_j$.

\item $d_j \le d_i \le s$: It holds that
\[ \rho_i = \max\{\kappa-\hat{c}_i,0\} \ge \max\{\kappa-\hat{c}_j,0\} = \rho'_j\, .\]
\end{enumerate}
As we have $\rho_i \ge \rho'_j$ in all cases, the claim also holds when $d_j \le d_i$, which completes the proof.
\end{enumerate}

\end{proof}

\begin{theorem}
If $\hat{\pmb{c}}$ or $\pmb{d}$ is a vector with constant values, an optimal solution to the balanced regret multi-representative selection problem can be found in polynomial time.
\end{theorem}

\begin{proof}
We apply the dominance property from Lemma~\ref{lemma:dominance}.
If $\hat{\pmb{c}}$ or $\pmb{d}$ is constant, an optimal solution can be found by selecting the $p_\ell$ items with smallest $\pmb{d}$ or $\hat{\pmb{c}}$ values for each partition $\ell \in [L]$, respectively.
\end{proof}
Note that the reductions used in the hardness results of Section~\ref{sec:selregrev} do not yield inapproximability bounds in $n$. This means that the existence of approximation algorithms remains a possibility. In particular, we show that we can determine in polynomial time if the optimal objective value is equal to zero, which is a prerequisite for the existence of approximation algorithms.

\begin{theorem}\label{Theo::0Solution}
For $\Gamma, \Gamma'\ge 1$, it is possible to determine in polynomial time if the objective value of the balanced regret problem is equal to zero. If this is the case, we can state an optimal solution.
\end{theorem}
\begin{proof}
Consider any solution $\pmb{x}\in\X$ that does not pack the $p_\ell$ cheapest items $i$ with respect to $\hat{c}_i + d_i$ for each set $T_\ell$. Then the adversary can construct a solution $\pmb{y}$ by choosing any such set $\ell$ and exchanging one item $i$ packed by $\pmb{x}$ for another item $j$ with smaller costs, i.e.,  $\hat{c}_j + d_j < \hat{c}_i + d_i$. All other items are packed as in $\pmb{x}$. As $\Gamma, \Gamma' \ge 1$, optimal adversarial and balancing strategies are to attack the one item, that is not packed by each respective solution. As $\hat{c}_j + d_j < \hat{c}_i + d_i$, the objective value is larger than zero. 

We conclude that if there exists a solution with objective value equal to zero, then this solution must pack the cheapest items with respect to $\hat{\pmb{c}} + \pmb{d}$. Due to the dominance criterion from Lemma~\ref{lemma:dominance}, such a solution can be found by lexicographically sorting each set primarily by $\hat{\pmb{c}} + \pmb{d}$ and secondarily by $\hat{\pmb{c}}$. The objective value of this solution can be checked by solving the adversarial problem \eqref{eq:revenge4}, which can be done in polynomial time due to Lemma~\ref{lemma:two}.
\end{proof}

We now consider the classic regret setting without balancing, i.e., $\Gamma'=0$. The complexity of this problem is currently open. We show that this case is solvable in polynomial time. Hence, in combination with the hardness of the general case, we see that the additional balancing stage does increase the complexity of our problem.

Consider the case $\Gamma'=0$ in formulation~\eqref{eq:revenge5}. By rewriting the problem, we find the following compact program for the regret problem.
\begin{subequations}
\label{eq:regrep}
\begin{align}
\min\ &\sum_{i\in[n]} \hat{c}_i x_i + \Gamma \pi + \sum_{i\in[n]} \rho_i - \sum_{\ell \in [L]} p_\ell \kappa_\ell \\
\text{s.t. } & \pi + \rho_i \ge d_i x_i & \forall i\in[n] \\
& \hat{c}_i + \rho_i \ge \kappa_\ell & \forall \ell\in[L], i\in T_\ell \\
& \sum_{i\in T_\ell} x_i = p_\ell & \forall \ell \in [L] \\
& x_i \in \{0,1\} & \forall i\in[n] \\
& \pi \ge 0 \\
& \kappa_\ell \ge 0 & \forall \ell \in [L] \\
& \rho_i \ge 0 & \forall i\in[n]
\end{align}
\end{subequations}
Note that we can assume $\rho_i = \max\{d_ix_i - \pi, \kappa_\ell - \hat{c}_i, 0\}$ in an optimal solution to problem~\eqref{eq:regrep}, where $i\in T_\ell$.

\begin{theorem}
Min-max regret multi-representative selection with budgeted uncertainty can be solved in $O(n^5)$.
\end{theorem}
\begin{proof}
Let us first assume that $\pmb{x}$ and $\pmb{\kappa}$ are fixed. The remaining problem only in variable $\pi$ is then to solve
\[ \min_{\pi \ge 0} \Gamma\pi + \sum_{\ell\in [L]}\sum_{i\in T_\ell} \max\{d_ix_i - \pi, \kappa_\ell - \hat{c}_i, 0\}\, . \]
Note that the objective function of this problem is piece-wise linear. Hence, an optimal value of $\pi$ is contained in the set of kink points, which is a subset of
\begin{align*}
P(\pmb{\kappa}) &= P_1 \cup P_2(\pmb{\kappa}) \\
\text{with } \qquad P_1 &= \{0\} \cup \{d_i : i\in[n]\} \\
P_2(\pmb{\kappa}) &= \{\hat{c}_i + d_i - \kappa_\ell : \ell\in[L], i\in T_\ell\} \, .
\end{align*}
Note that the set $P(\pmb{\kappa})$ does not depend on the choice of $\pmb{x}$.

Now let us assume that $\pmb{x}$ and $\pi$ are fixed. Then it is possible to split problem~\eqref{eq:regrep} into independent subproblems. That is, for each set $\ell\in[L]$, the remaining problems only in $\kappa_\ell$ are of the form
\[ 
\min_{\kappa_\ell \ge 0} \sum_{i\in T_\ell} \max\{d_ix_i - \pi, \kappa_\ell - \hat{c}_i, 0\} - p_\ell \kappa_\ell \, . \]
Note that this problem is piecewise-linear in $\kappa_\ell$. Hence, there exists an optimal solution where $\kappa_\ell$ is at one of the kink points, which are contained in the set
\[ K_\ell(\pi) = \{0\}\cup \{ \hat{c}_i : i\in T_\ell\} \cup \{ \hat{c}_i - \pi : i\in T_\ell \} \cup \{ \hat{c}_i + d_i - \pi : i\in T_\ell \} \, .\]
We can make the following case distinctions.
\begin{enumerate}
\item First consider the case that we choose $\pi\in P_1=\{0\} \cup \{d_i : i\in[n]\}$. Then, we decompose problem~\eqref{eq:regrep} into independent subproblems for each $\ell\in[L]$. For each subproblem, there are $O(|T_\ell|)$ many possible values for $\kappa_\ell$ to enumerate. For each choice of $\kappa_\ell$, the remaining problem in $\pmb{x}$ only can be solved in $O(|T_\ell|)$. As there are $O(n)$ many values for $\pi$ to check, this case requires a total time in $O(n\cdot \sum_{\ell\in[L]} |T_\ell|^2) = O(n^3)$.

\item Now consider the case that we want to choose some $\pi \in P_2(\pmb{\kappa}) = \{\hat{c}_i + d_i - \kappa_\ell : \ell\in[L], i\in T_\ell\}$. We model this choice through the index of item that defines $\pi$, that is, we set $\pi = \hat{c}_k + d_k - \kappa_j$ for a specific choice of $j\in[L]$ and $k\in T_j$. Let us assume for now that $\kappa_j$ is fixed. For each $\ell\neq j$ the subproblem for arbitrary $x\in \{0,1\}^n$ becomes
\[ \min_{\kappa_\ell \ge 0} \sum_{i\in T_\ell} \max\{d_ix_i - (\hat{c}_k + d_k - \kappa_j), \kappa_\ell - \hat{c}_i, 0\} - p_\ell \kappa_\ell \, .\]
Hence, an optimal choice for $\kappa_\ell$ is contained in the set
\begin{align*}
K_\ell &= K_1 \cup K_2 \\
\text{with } \qquad
K_1 &= \{0\} \cup \{\hat{c}_i : i\in T_\ell\} \\
K_2 &= \{\hat{c}_i + d_i - \hat{c}_k - d_k + \kappa_j  : i\in T_\ell\}\, .
\end{align*}
Let $U\subseteq[L]\setminus\{j\}$ be the set of indices where we choose $\kappa_\ell \in K_1$ and $V\subseteq[L]\setminus\{j\}$ be the set of indices where we choose $\kappa_\ell \in K_2$. In the first case, we write $\kappa_\ell = u_\ell$ for some constant $u_\ell$ and in the second case, we write $\kappa_\ell = v_\ell + \kappa_j$ for some constant $v_\ell$.
The problem in $\kappa_j$ is then
\begin{align*}
\min_{\kappa_j \ge 0}\ &\sum_{i\in T_j} \max\{d_ix_i - (\hat{c}_k + d_k - \kappa_j), \kappa_j - \hat{c}_i, 0\} - p_j \kappa_j \\
+ &\sum_{\ell \in V} \sum_{i\in T_\ell} \max\{d_ix_i - (\hat{c}_k + d_k - \kappa_j), v_\ell + \kappa_j - \hat{c}_i , 0\} \, .
\end{align*}
Again, this is a piecewise linear optimization problem, where an optimal solution is at one of the kink points, contained in
\begin{align*}
 K_j =& \{0\} \cup  \{\hat{c}_i : i\in T_j\} 
 \cup \{ \hat{c}_i - v_\ell : \ell \in V, i\in T_\ell \} 
 \cup \{ \hat{c}_k + d_k - d_i : i\in [n] \}\\
 \subseteq& \{0\} \cup  \{\hat{c}_i : i\in T_j\} 
 \cup \{ \hat{c}_i -\hat{c}_{i'}- d_{i'} +\hat{c}_k+d_k : \ell \in [L], i,i'\in T_\ell \} \\
& \cup \{ \hat{c}_k + d_k - d_i : i\in [n] \}  \, .\end{align*}
In total, there are $O(n)$ possible values to choose $j$ and $k$ to determine $\pi$. There are $O(n^2)$ possible values to determine $\kappa_j$. The subproblem in each $\ell\neq j$ can then be solved in $O(|T_\ell|^2)$. Hence, this case can be solved in $O(n^3 \cdot \sum_{\ell\in[L]} |T_\ell|^2) = O(n^5)$.
\end{enumerate}
\end{proof}

\section{Experiments}
\label{sec:experiments}

In order to demonstrate the practical implications of introducing the balancing option we perform experiments using selection and knapsack problems with randomly created data, and shortest path problems with real-world data. We examine the computational tractability by comparing different modeling and solving techniques and illustrate the added value of the extension by showing how such solutions can yield a trade-off between solutions of the worst case approach (with interval uncertainty) and solution that are optimal with respect to a minimal regret with $\Gamma$ uncertainty. 

We make use of the following solution methods (see Section~\ref{SubSec::SolutionMethods}): the enumeration approach (where all scenarios are used in one model), the iterative method (where scenarios are generated one by one), solving the problem as a quantified integer program with Yasol, and solving a compact mixed-integer programming formulation, if available.
For solving mixed-integer programs, we use CPLEX 12.9.0 restricted to a single thread, but otherwise default settings. For solving the quantified programming formulation we use the latest version of the open-source solver Yasol running CPLEX 12.6.1 as its LP solver. 
All experiment were executed on a desktop computer with AMD Ryzen 9 5900X processors and 128GB RAM with a time limit of 1800 seconds per instance.

\subsection{Selection Problem}
First, we investigate the selection problem, i.e. $\mathcal{X}=\{\pmb{x} \in \{0,1\}^n \mid \sum_{i \in [n]}x_i=p\}$, with $p=\lfloor \frac{n}{2}\rfloor$. Instances are created by selecting the nominal cost $\hat{c}_i$  for each item $i$ uniformly random (u.r.) from the set $\{1,\ldots,100\}$. The additional cost $d_i$ is selected u.r. from $\{0,\ldots,99\}$. 

We are interested in the runtimes of the four available solution techniques, which are solving the fully enumerated robust counterpart, solving the quantified programming formulation, using the iterative method 
 and solving the compact formulation. In Figure~\ref{Fig::TimePlotSelectionGammaConst} 
we show for fixed $\Gamma=2$ and $\Gamma'=1$ the median of the runtimes for each solution techniques for $50$ instances per $n\in\{8,10,\ldots,20,30,\ldots, 200\}$.

\begin{figure}[h!]
\centering
\includegraphics[scale=.6]{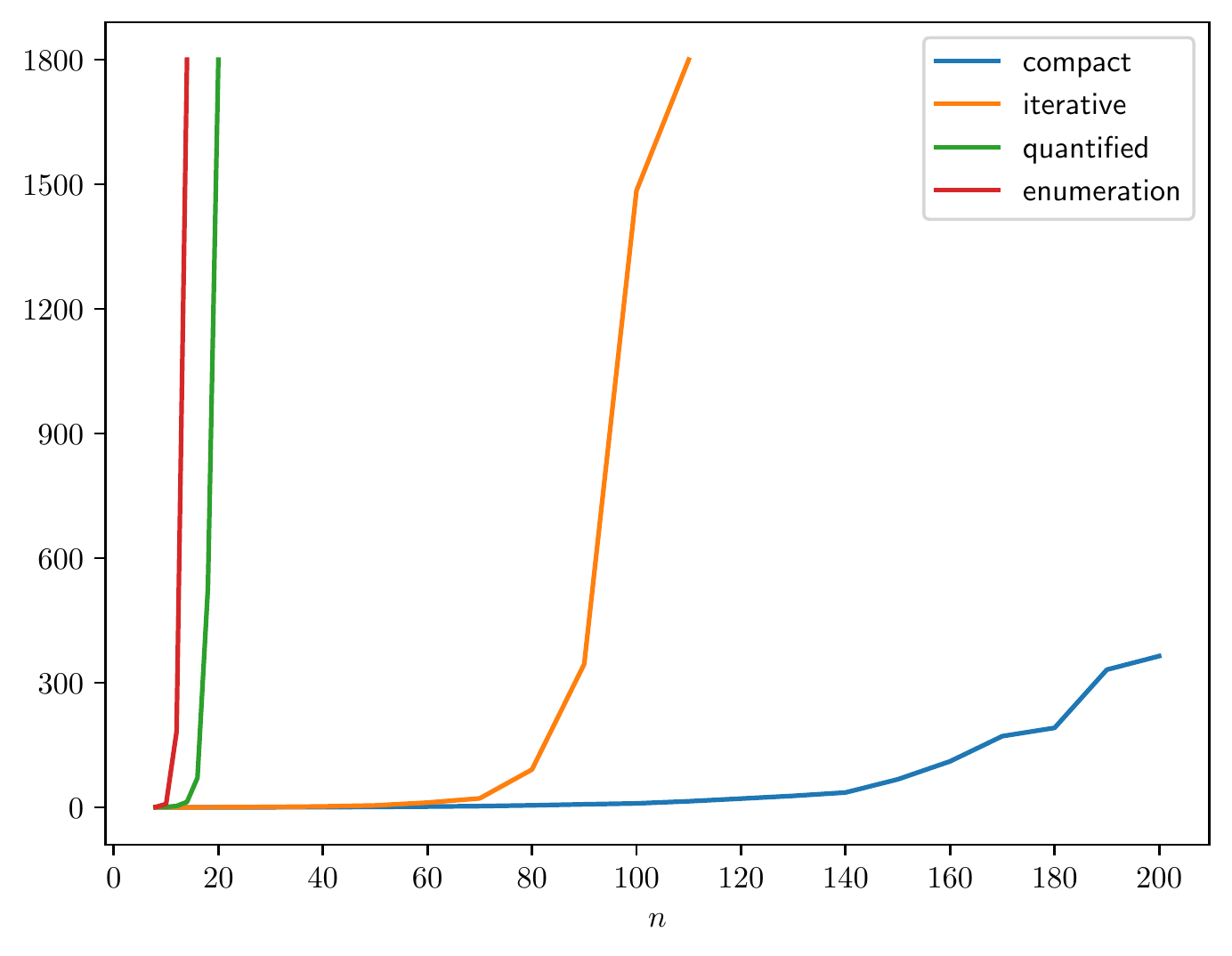}
\caption{Median runtimes on selection instances with $\Gamma=2$ and $\Gamma'=1$.\label{Fig::TimePlotSelectionGammaConst}}
\end{figure}
The two general model-and-run approaches quickly hit the timelimit: For $n=14$ and $n=20$ most of the robust enumeration and quantified instances, respectively, cannot be solved within half an hour. As expected, the compact formulation can be solved most efficiently and even for $n=200$ the runtime of most instances remains far below the timelimit. The iterative method still can deal with instances with up to $n=100$ but then more than half of the instances also exceed the timelimit.

For $n \in \{20,\ldots,200\}$ we examine how the runtimes of the compact and iterative approach scale for non-constant $\Gamma=\frac{n}{5}$ and $\Gamma'=\frac{n}{10}$. In Table~\ref{Tab::SelGammaO} the median and mean runtimes in seconds, and percentage of solved instances for the compact and iterative approach is shown. Furthermore, we display the median runtimes and the number of solved instances dependent on their objective value for the compact and iterative approach, respectively. 
\begin{table}[h!]
\centering
\footnotesize
\caption{Data on selection instances with $\Gamma=\frac{n}{5}$ and $\Gamma'=\frac{n}{10}$.\label{Tab::SelGammaO}}
\begin{tabular}{r|rrrrrrr|rrrrr}
     & \multicolumn{7}{c|}{compact}      & \multicolumn{5}{c}{iterative}  \\
     &         &    & &    \multicolumn{2}{c}{solved}        & \multicolumn{2}{c|}{median}        &           &        &        & \multicolumn{2}{c}{solved}         \\
 $n$ & median  & mean  & solved & $>0$   & $=0$ &   $>0$  & $=0$ & median    & mean   & solved & $>0$   & $=0$ \\\midrule
20 & 0.0 & 0.1 & 1.00 & 31 & 19 & 0.1 & 0.0 & 0.3 & 1.0 & 1.00 & 31 & 19 \\
30 & 0.1 & 0.2 & 1.00 & 31 & 19 & 0.2 & 0.0 & 3.0 & 33.2 & 1.00 & 31 & 19 \\
40 & 0.2 & 0.3 & 1.00 & 29 & 21 & 0.4 & 0.1 & 22.2 & 275.3 & 0.94 & 26 & 21 \\
50 & 0.5 & 1.1 & 1.00 & 31 & 19 & 1.5 & 0.2 & 1800.0 & 1046.6 & 0.44 & 3 & 19 \\
60 & 1.6 & 2.8 & 1.00 & 36 & 14 & 2.5 & 0.3 & 1800.0 & 1286.9 & 0.30 & 1 & 14 \\
70 & 2.4 & 5.0 & 1.00 & 32 & 18 & 3.9 & 0.5 & 1800.0 & 1167.6 & 0.38 & 1 & 18 \\
80 & 4.2 & 5.7 & 1.00 & 32 & 18 & 6.6 & 0.7 & 1800.0 & 1166.0 & 0.36 & 0 & 18 \\
90 & 3.6 & 13.8 & 1.00 & 27 & 23 & 12.6 & 1.2 & 1800.0 & 1010.9 & 0.46 & 0 & 23 \\
100 & 5.7 & 18.5 & 1.00 & 30 & 20 & 15.2 & 1.8 & 1800.0 & 1172.9 & 0.36 & 0 & 18 \\
110 & 9.5 & 31.0 & 1.00 & 30 & 20 & 23.2 & 2.2 & 1800.0 & 1379.5 & 0.26 & 0 & 13 \\
120 & 18.7 & 41.1 & 1.00 & 34 & 16 & 36.8 & 2.7 & 1800.0 & 1537.2 & 0.24 & 0 & 12 \\
130 & 27.1 & 58.6 & 1.00 & 33 & 17 & 51.3 & 3.9 &  -       &  -       &  -     &  -   &  -   \\
140 & 26.9 & 82.4 & 0.98 & 30 & 19 & 55.4 & 4.3 &  -       &  -       &  -     &  -   &  -   \\
150 & 44.1 & 108.8 & 1.00 & 30 & 20 & 135.4 & 5.7 &  -       &  -       &  -     &  -   &  -   \\
160 & 67.1 & 150.8 & 1.00 & 38 & 12 & 122.5 & 11.2 &  -       &  -       &  -     &  -   &  -   \\
170 & 141.3 & 206.8 & 1.00 & 42 & 8 & 175.1 & 6.9 &  -       &  -       &  -     &  -   &  -   \\
180 & 58.6 & 135.0 & 1.00 & 33 & 17 & 123.2 & 10.8 &  -       &  -       &  -     &  -   &  -   \\
190 & 75.7 & 209.1 & 1.00 & 35 & 15 & 161.1 & 12.4 &  -       &  -       &  -     &  -   &  -   \\
200 & 199.5 & 307.8 & 0.98 & 37 & 12 & 353.7 & 14.9 &  -       &  -       &  -     &  -   &  -  
\end{tabular}
\end{table}
While the ratio of instances with non-zero and zero nearly stays the same for increasing $n$, the number of solved instance with non-zero objective value using the iterative method quickly decreases. Similarly, when solving the compact formulation, instances with non-zero objective take considerably more time to be solved. While the overall median runtime of the compact formulation seems to scale better compared to constant $\Gamma$ and $\Gamma'$, when comparing the median runtime on the non-zero instances, the results are similar to the ones shown in Figure~\ref{Fig::TimePlotSelectionGammaConst}.  
\begin{figure}[h!]
\centering
\includegraphics[width=.49\textwidth]{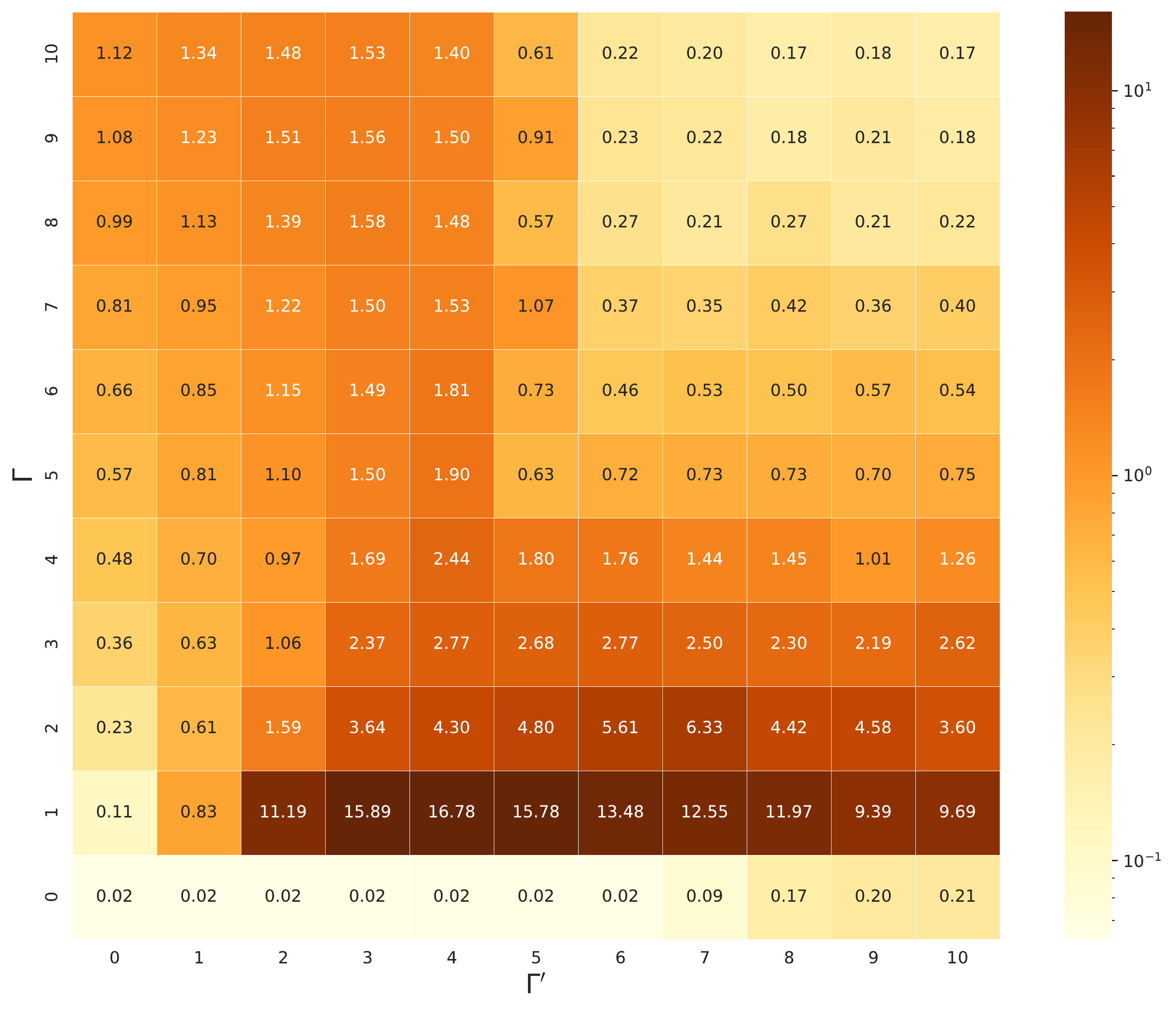} \includegraphics[width=.49\textwidth]{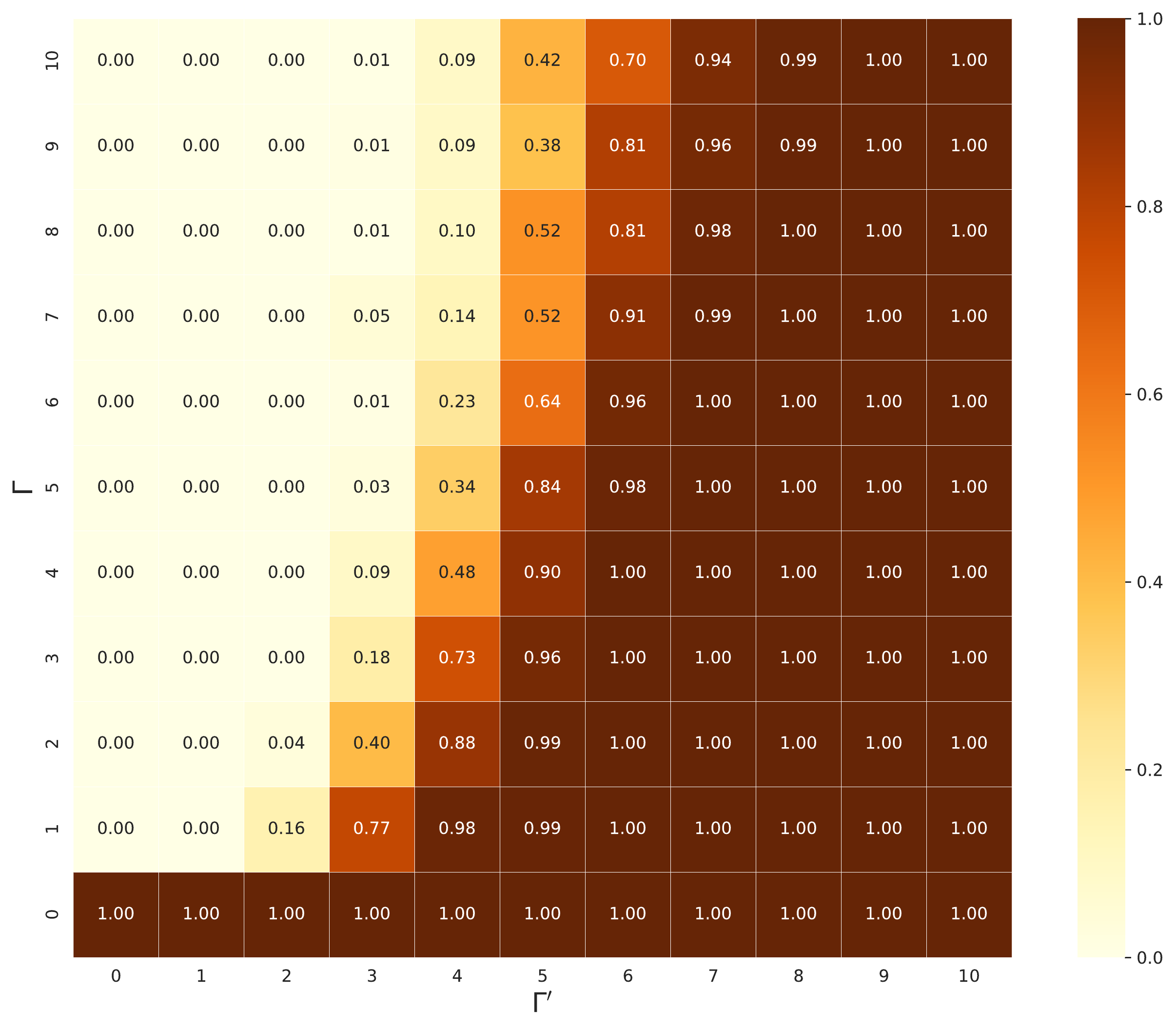}

\caption{Median runtime on the compact formulation of selection instances with $n=50$ and various $\Gamma$ and $\Gamma'$ (left) and the percentage of instances with objective value $0$ (right).\label{Fig::SelectionN50}}
\end{figure}

In order to grasp what constellations of $\Gamma$ and $\Gamma'$ are particularly hard to solve, we set $n=50$ and let  $\Gamma$ and $\Gamma'$ take all values from $0$ to $10$. For each constellation we run $200$ instances of the compact formulation and show the respective median runtime in Figure~\ref{Fig::SelectionN50} (left). All instances were solved to optimality. On the right-hand side of this figure, the percentage of instances with objective value of $0$ are shown. Even for $\Gamma<\Gamma'$ we find several instances that have a non-zero objective value. 
Interestingly, 
the most difficult instances for CPLEX have
$1=\Gamma<\Gamma'$, even though for $\Gamma'\geq 4$, almost all instances have an objective value of zero. Here our findings of Theorem~\ref{Theo::0Solution} can come to a practical use, as we know that for such instances, a polynomial time algorithm exists, which means that the solution by CPLEX is not necessary in the first place.

Furthermore, we want to understand the performance of the balanced regret approach with respect to other optimization approaches. To this end, we evaluate optimal solutions of the balanced regret approach (\RR{$\Gamma'$}) using five different evaluation criteria: The best case case (\BC) that only uses the nominal costs $\hat{\pmb{c}}$, the worst-case optimization approach (with interval uncertainty) (\WCI) that assumes worst case costs $\hat{\pmb{c}}+\pmb{d}$ for every selected item, the worst-case optimization approach with budgeted uncertainty (\WCGamma), allowing at most $\Gamma$ items to become expensive,  regret minimization (with interval uncertainty) (\RegretI) where all selected items have worst case costs $\hat{\pmb{c}}+\pmb{d}$, and the regret approach with $\Gamma$ uncertainty (\RegretGamma). Note that \RR{0} is the same as \RegretGamma. In Table~\ref{Tab::QualitySelection-left} and Figure~\ref{Tab::QualitySelection-right} the quality of the different robust optimization frameworks is displayed by showing the average relative difference (percentage change) of the solution within each framework compared to the respective optimal solution over $1000$ instances with $n=60$ and $\Gamma=15$. The rows indicate the used solution methods and the columns the evaluation criteria.   

\begin{table}[htb]
\centering
\footnotesize
\caption{Average relative difference of the different optimization approaches on selection instances compared to the optimal solutions of \BC \WCI, \WCGamma, \RegretI and \RegretGamma. \label{Tab::QualitySelection-left}}
\begin{tabular}{llllll}
          & \BC    & \WCI & \WCGamma & \RegretI & \RegretGamma \\
          \midrule
\BC        & 0.000 & 0.124 & 0.062     & 0.233     & 0.228  \\
\WCI     & 0.322 & 0.000 & 0.015     & 0.205     & 0.203  \\
\WCGamma & 0.178 & 0.018 & 0.000     & 0.098     & 0.098  \\
\RegretI & 0.085 & 0.030 & 0.013     & 0.000     & 0.000  \\
\RegretGamma    & 0.084 & 0.030 & 0.013     & 0.000     & 0.000  \\
\RR{1}     & 0.086 & 0.029 & 0.013     & 0.001     & 0.000  \\
\RR{2}     & 0.090 & 0.028 & 0.012     & 0.002     & 0.002  \\
\RR{3}    & 0.098 & 0.026 & 0.011     & 0.005     & 0.005  \\
\RR{4}     & 0.115 & 0.022 & 0.009     & 0.018     & 0.018  \\
\RR{5}     & 0.156 & 0.015 & 0.009     & 0.052     & 0.052  \\
\RR{6}     & 0.224 & 0.007 & 0.010     & 0.114     & 0.113  \\
\RR{7}     & 0.281 & 0.002 & 0.013     & 0.168     & 0.167  \\
\RR{8}    & 0.308 & 0.000 & 0.014     & 0.191     & 0.190  \\
\RR{9}   & 0.318 & 0.000 & 0.015     & 0.200     & 0.199  \\
\RR{10}   & 0.320 & 0.000 & 0.015     & 0.202     & 0.200
\end{tabular}
\end{table}

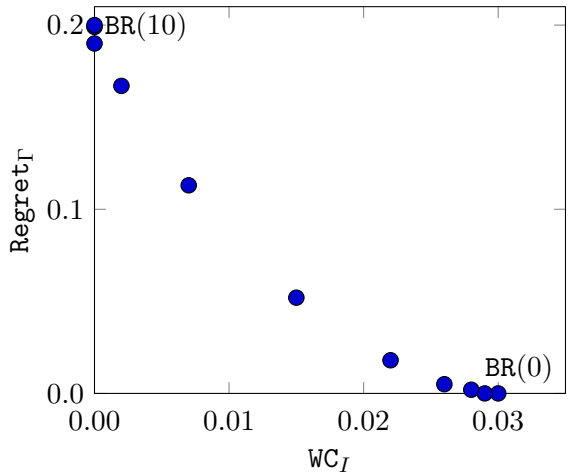
\begin{figure}[htb]
\centering
\begin{tikzpicture}

\begin{axis}[
  ylabel = {\RegretGamma},
  ylabel shift = -3 pt,
    xlabel = {\WCI},
    xtick distance={0.01},
    ymax=0.21,
    xmax=0.035,
     width=.5\textwidth,
    ytick distance={0.1},
    enlargelimits=false,
       y tick label style={
        /pgf/number format/.cd,
            fixed,
            fixed zerofill,
            precision=1,
        /tikz/.cd,
    },
    x tick label style={
        /pgf/number format/.cd,
            fixed,
            fixed zerofill,
        /tikz/.cd,
        below=1mm
    },
    scaled ticks=false
]
\addplot+[
    only marks,
    color=black,
    scatter,
    mark=*,
    mark size=2.9pt, 
    point meta=explicit symbolic,
    visualization depends on=\thisrow{alignment} \as \alignment,
    visualization depends on=\thisrow{xshift} \as \xshift,
    visualization depends on=\thisrow{yshift} \as \yshift,
    nodes near coords,
     every node near coord/.style={anchor=\alignment, xshift=\xshift, yshift=\yshift}]
table[meta=pin]
{selectionScatter.dat};
\end{axis}
\end{tikzpicture}
\captionof{figure}{Evaluation of the \RR{$\Gamma'$} selection solutions with respect to \WCI and  \RegretGamma.
\label{Tab::QualitySelection-right}}
\end{figure}

It can be seen that for increasing $\Gamma'$ the optimal solutions of \RR{$\Gamma'$} also become optimal with respect to \WCI. Note that the row of \RR{10} is almost identical to the row corresponding to the original \WCI solution. The \BC, \RegretI, and \RegretGamma, evaluation of \RR{$\Gamma'$} become worse for increasing $\Gamma'$. But the evaluation values from \RR{0} to \RR{10} do not behave monotonously in general: the \WCGamma value of the \RR{$\Gamma'$} solutions first decreases, before it increases up to the \WCGamma value of \WCI. 
To illustrate the trade-off the balanced regret approach constitutes for \WCI and \RegretGamma, the scatter plot of the \WCI and \RegretGamma evaluations of all \RR{$\Gamma'$} solutions is shown in Figure~\ref{Tab::QualitySelection-right}.

\subsection{Knapsack}
We also consider knapsack instances, i.e. $\mathcal{X}= \{\pmb{x}\in\{0,1\}^n \mid \sum_{i \in [n]} w_ix_i \leq C\}$. We create \textit{almost strongly correlated} instances as proposed in \cite{furini2015heuristic}. The item weights $w_i$ of the items are selected u.r. from $\{1,\ldots,\bar{R}\}$, with $\bar{R}=1000$. The nominal profit $\hat{c}_i$ of item $i$ is then selected u.r. from $\{\lceil 0.8p_i \rceil,\ldots,\lceil p_i \rceil \}$ with $p_i$ u.r. from $\{w_i + \frac{\bar{R}}{10}-\frac{\bar{R}}{500},w_i + \frac{\bar{R}}{10}+\frac{\bar{R}}{500} \}$. The potential decrease $d_i$ in the profit of item $i$ is then selected u.r. from $\{\lceil p_i \rceil-\hat{c}_i,\ldots,\lceil 1.2p_i \rceil -\hat{c}_i \}$. 

Again, we first consider the three available solution approaches of solving the fully enumerated counterpart, the quantified program as well as using the iterative solution method. In Table~\ref{Tab::KnapsackTo18} for $n \in\{6,\ldots,18\}$, $\Gamma=2$ and $\Gamma'=1$ the median runtimes are shown with $50$ instances per entry. While solving the quantified program requires less time than solving the fully enumerated counterpart, both approaches quickly hit the timelimit of half an hour. For the iterative method, on the other hand, even the longest runtime for instances with $n<20$ is below two seconds. Hence, in the following experiments we only use the iterative method. 
\begin{table}[h!]
\centering
\footnotesize
\caption{Median runtime of knapsack instances with $\Gamma=2$ and $\Gamma'=1$ in seconds.\label{Tab::KnapsackTo18}}
\begin{tabular}{lrrr}
 $n$  & iterative   & quantified     & enumeration       \\\midrule
6  & 0.01 & 0.06    & 0.06    \\
8  & 0.03 & 0.72    & 1.24    \\
10 & 0.04 & 1.70    & 17.07   \\
12 & 0.07 & 9.61    & 270.29  \\
14 & 0.07 & 52.44   & 1800.00 \\
16 & 0.13 & 315.74  &   -      \\
18 & 0.18 & 1800.00 &   -   
\end{tabular}
\end{table}

To understand the impact of different values of $\Gamma$ and $\Gamma'$, we fix $n=40$ and vary $\Gamma,\Gamma' \in \{0,1,\ldots,10\}$. For each constellation we solve $200$ instances using the iterative method and show the respective median runtime in Figure~\ref{Fig::KnapsackN40} (left). On the right-hand side of this figure the percentage of instances with objective value of $0$ are shown. For the iterative method instances with large $\Gamma$ that have a non-zero objective value are particularly hard to solve. Furthermore, for $\Gamma \in \{2,3\}$, even though the objective value is often zero for $\Gamma'>\Gamma$, the runtime is slightly increased in comparison to other values of $\Gamma$. 

\begin{figure}[h!]
\centering
\includegraphics[width=.49\textwidth]{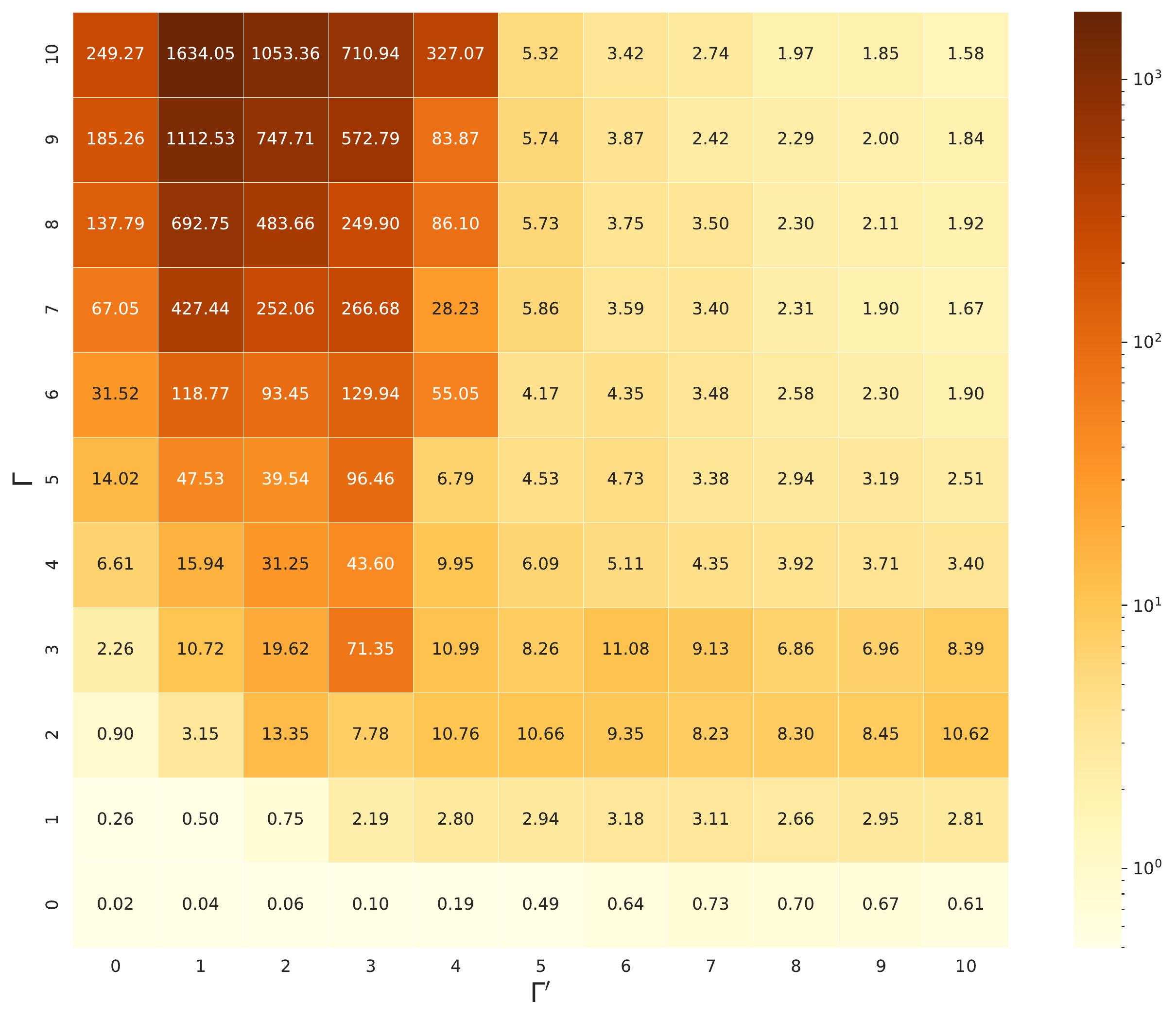} \includegraphics[width=.49\textwidth]{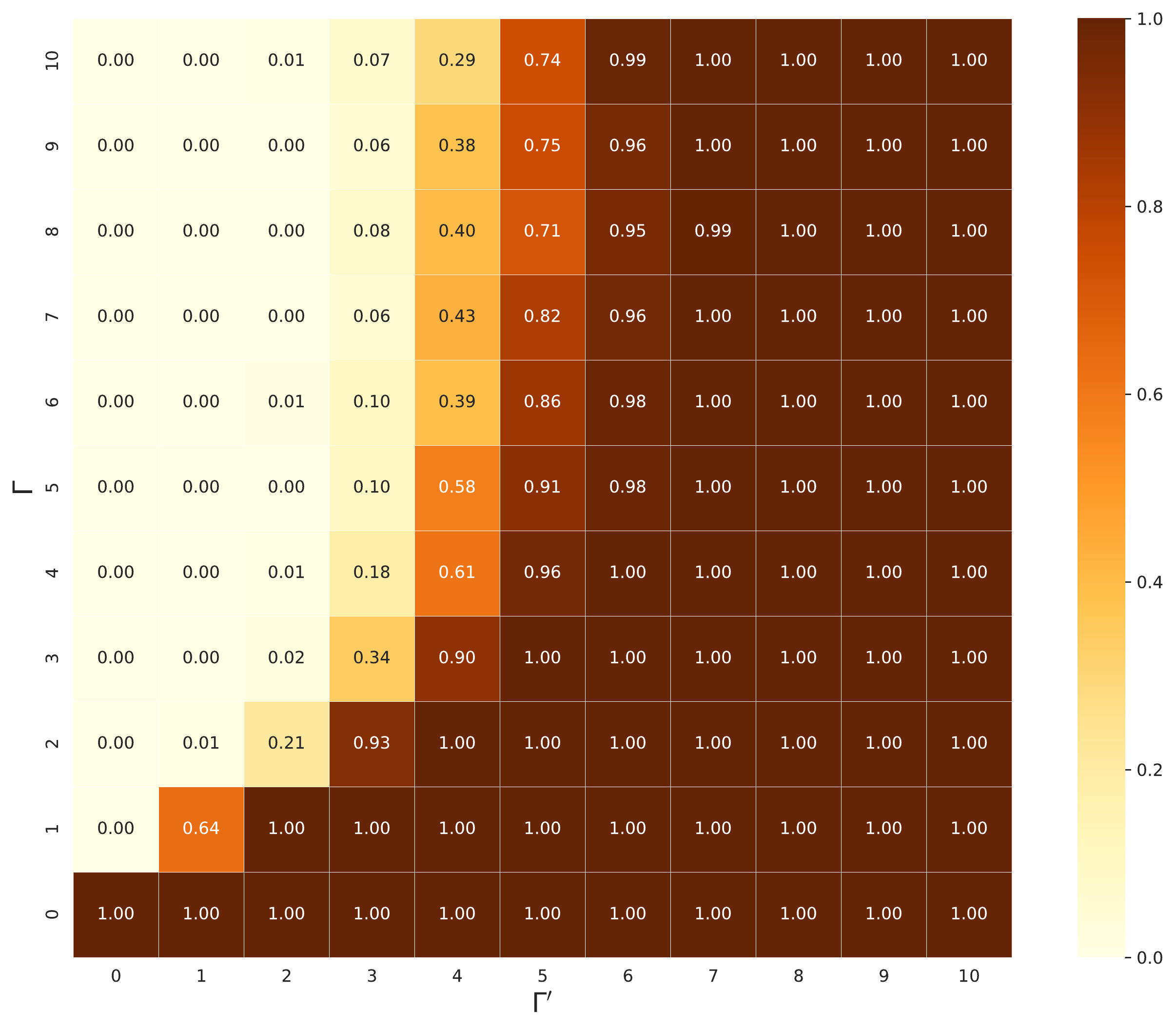}

\caption{Median runtime of the iterative method on knapsack instances with $n=40$ and various $\Gamma$ and $\Gamma'$ (left) and the percentage of instances with objective value $0$ (right).\label{Fig::KnapsackN40}}
\end{figure}

In the next experiment we investigate how the iterative method scales for increasing $n$ on each $200$ instances with $\Gamma=2$ and $\Gamma'=1$ and also $\Gamma=\frac{n}{5}$ and $\Gamma'=\frac{n}{10}$. Note that in the latter setting, slightly changing  the ratio between $\Gamma$, $\Gamma'$ and $n$ can have a considerable impact on the results (cf. Figure~\ref{Fig::KnapsackN40}(left)). 
\begin{table}[h!]
\centering
\footnotesize
\caption{Data on knapsack instances solved using the iterative method. 
\label{Tab::KnapsackIter}}
\begin{tabular}{rrrrrrrrrrrrrr}
$\Gamma$&$\Gamma'$&$n$&20&30&40&50&60&70&80&90&100\\\midrule
\multirow{3}{*}{2}&\multirow{3}{*}{1}&Median & 0.2  & 0.8  & 2.9   & 8.8    & 23.3   & 54.3   & 100.3  & 192.2  & 534.8  \\
&&Mean & 0.3  & 1.2  & 4.9   & 14.9   & 47.8   & 110.1  & 199.7  & 390.2  & 769.0  \\
&&Solved& 1.00 & 1.00 & 1.00  & 1.00   & 1.00   & 1.00   & 1.00   & 0.94   & 0.81   
\\
&&zero& 17	 &3	&0	&0	&0	&0	&0	&0	&0	\\
&& non-zero &183	&197&	200	&200	&200&	200	&200	&188&	162
\\\midrule
\multirow{5}{*}{$\frac{n}{5}$}&\multirow{5}{*}{$\frac{n}{10}$}&Median   & 0.3  & 3.5  & 169.8 & 1800.0 & 1800.0 & 1800.0 & 1800.0 & 1800.0 & 1800.0 \\
&&Mean    & 0.7  & 40.6 & 632.3 & 1053.3 & 1141.3 & 1195.4 & 1142.0 & 1273.2 & 1167.9 \\
&&Solved  & 1.00 & 1.00 & 0.77  & 0.44   & 0.39   & 0.37   & 0.42   & 0.34   & 0.39 \\
&&  zero &56&	72&	65	&79	&78&	74&	83&	68	&78\\
&& non-zero &144&	128	&88	&8&	0	&0	&0&	0	&0 
\end{tabular}
\end{table}
In Table \ref{Tab::KnapsackIter} the median and mean runtimes are shown, as well as the percentage of solved instances.  We additionally show how many of the instances solved within the timelimit have a  non-zero and zero objective value.  For $\Gamma$ and $\Gamma'$ proportionally growing with $n$ the runtimes as well as the percentage of solved instances quickly worsen compared to the instances with constant $\Gamma$ and $\Gamma'$. However, for constant $\Gamma$ only few instances (with $n=20$ and $n=30$) have non-zero solutions. For $\Gamma=\frac{n}{5}$ several solved instances with $n\leq 50$ have a non-zero objective, but for $n\geq 60$ all solved instances have an objective value of zero. Comparing this to the experiments on the selection problem it is probable that at least some of the instances that were not solved have non-zero objective values and are harder to solve via the iterative method, as already suggested by Figure \ref{Fig::KnapsackN40}.

Similar to the experiments for the selection problem we want to view the solutions in the context of other  optimization approaches. For $n=40$ and $\Gamma=6$ the relative differences of the solutions with respect to the different evaluation criteria are shown in Table~\ref{Tab::QualityKnapsack}. Again, for increasing $\Gamma'$, the \RR{$\Gamma'$} solution converges to an optimal \WCI solution. Figure~\ref{Fig::ScatterKnapsack} illustrates that also for the knapsack problem balanced regret constitutes a trade-off between \WCI and \RegretGamma.

\begin{table}[h!]
\centering
\footnotesize
\caption{Average relative difference of the different optimization approaches on  
knapsack instances 
compared to the optimal solutions of \BC \WCI, \WCGamma, \RegretI and \RegretGamma.\label{Tab::QualityKnapsack}}
\begin{tabular}{llllll}
              & \BC    & \WCI & \WCGamma & \RegretI & \RegretGamma \\\midrule
\BC            &  0.000 & 0.020 & 0.007 & 0.169 & 0.133 \\
\WCI         &0.010 & 0.000 & 0.009 & 0.055 & 0.169 \\
\WCGamma       &0.004 & 0.009 & 0.000 & 0.071 & 0.028 \\
\RegretI     &0.005 & 0.003 & 0.004 & 0.000 & 0.045 \\
\RegretGamma &0.003 & 0.007 & 0.001 & 0.036 & 0.000 \\
\RR{1}         &0.003 & 0.007 & 0.001 & 0.036 & 0.004 \\
\RR{2}         &0.004 & 0.006 & 0.001 & 0.032 & 0.010 \\
\RR{3}       &0.005 & 0.004 & 0.002 & 0.035 & 0.035 \\
\RR{4}        &0.008 & 0.001 & 0.005 & 0.033 & 0.100 \\
\RR{5}         &0.009 & 0.000 & 0.008 & 0.046 & 0.145 
\end{tabular}
\end{table}

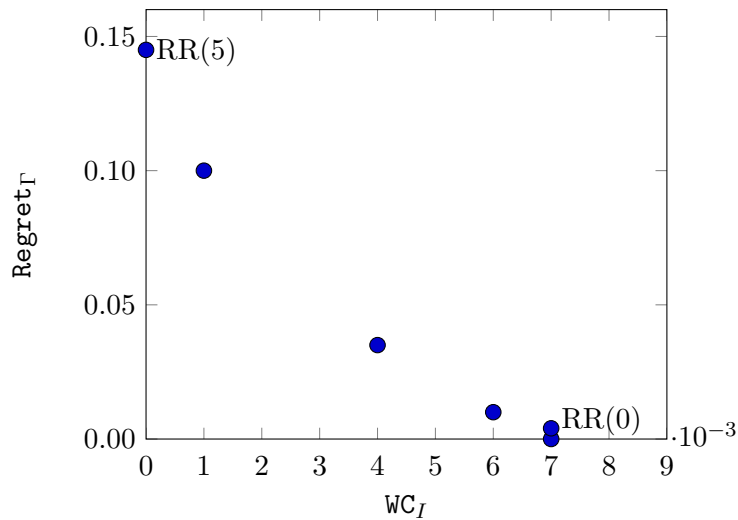
\begin{figure}[h!]
\centering
\begin{tikzpicture}

\begin{axis}[
  ylabel = {\RegretGamma},
    ylabel shift = -3 pt,
    xlabel = \WCI,
    ymin=0,
    ymax=0.16,
    xmax=0.009,
        enlargelimits=false,
        ylabel shift = 10 pt,
            xtick distance={0.001},
    ytick distance={0.05},
        y tick label style={
        /pgf/number format/.cd,
            fixed,
            fixed zerofill,
            precision=2,
        /tikz/.cd,
    },
    x tick label style={
        /pgf/number format/.cd,
            fixed,
            fixed zerofill,
            precision=0,
        /tikz/.cd,
              below=1mm
    },
    every x tick scale label/.style={
    at={(1,0)},xshift=1pt,anchor=south west,inner sep=0pt
}
]
\addplot+[
    only marks,
    color=black,
    scatter,
    mark=*,
    mark size=2.9pt, 
    point meta=explicit symbolic,
    visualization depends on=\thisrow{alignment} \as \alignment,
    nodes near coords,
     every node near coord/.style={anchor=\alignment}]
table[meta=pin]
{knapsackScatter.dat};
\end{axis}
\end{tikzpicture}
\caption{Evaluation of the \RR{$\Gamma'$} knapsack solutions with respect to \WCI and \RegretGamma.
 \label{Fig::ScatterKnapsack}}
\end{figure}

\subsection{Shortest Path}
For a third experiment we consider the shortest path problem $\mathcal{X}=\{\pmb{x} \in \{0,1\}^E \mid  \sum_{e \in \delta^-(v)} x_e - \sum_{e \in \delta^+(v)} x_e = \mathds{1}_t(v)-\mathds{1}_s(v)\ \forall v\in V\}$ with underlying graph $G=(V,E)$, $\mathds{1}$ being the indicator function and $\delta^-(v)$ and $\delta^+(v)$ the set of ingoing and outgoing arcs of node $v$, respectively. We use real world data introduced in \cite{chassein2019algorithms} and provided by the city of Chicago with a graph containing 538 nodes and 1308 arcs.  
4363 scenarios are obtained containing the traversal times for arcs. For each arc $e$ we use the first decile as the nominal travel time $\hat{c}_e$ and by calculating the difference to the last decile we obtained the additional travel time $d_e$. We generated 200 random $(s,t)$ pairs and this way obtained $200$ shortest path instances.

Using only the iterative solution method, for $\Gamma,\Gamma' \in \{0,1,\ldots,10\}$ the median runtimes are illustrated in Figure~\ref{Fig::SPN50} (left). Additionally, in Figure~\ref{Fig::SPN50} (right), we show the percentage of solutions having an objective value of $0$. While the overall runtimes are rather small, they tend to be higher for larger $\Gamma$ with small $\Gamma'<\Gamma$. This coincides with the results obtained for the knapsack instances. However, unlike in the experiments on selection and knapsack instances, here the boundary where (almost) all instances have an objective value of zero proceeds more strictly at $\Gamma=\Gamma'-1$.
\begin{figure}[h!]
\centering
\includegraphics[width=.49\textwidth]{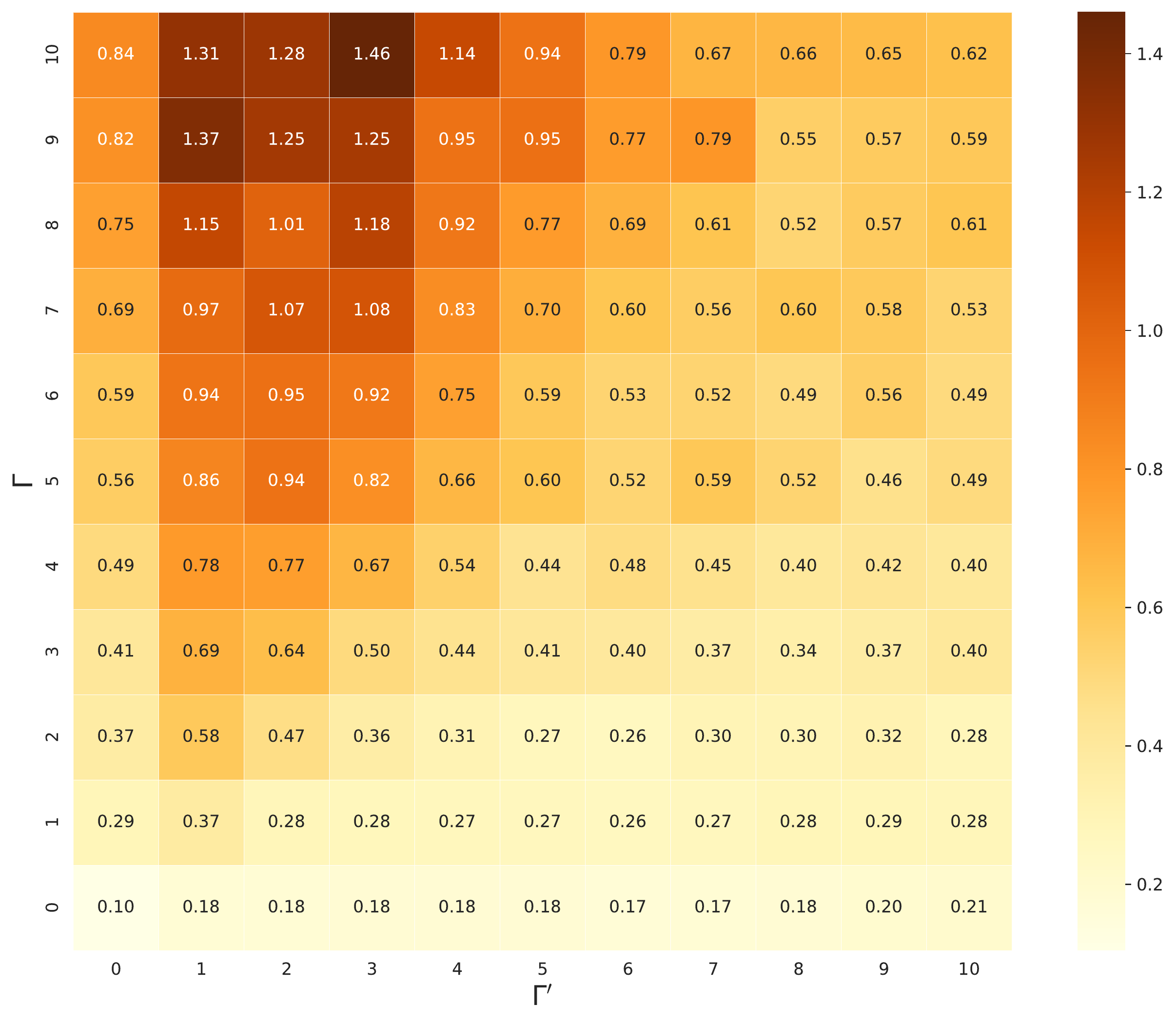} \includegraphics[width=.49\textwidth]{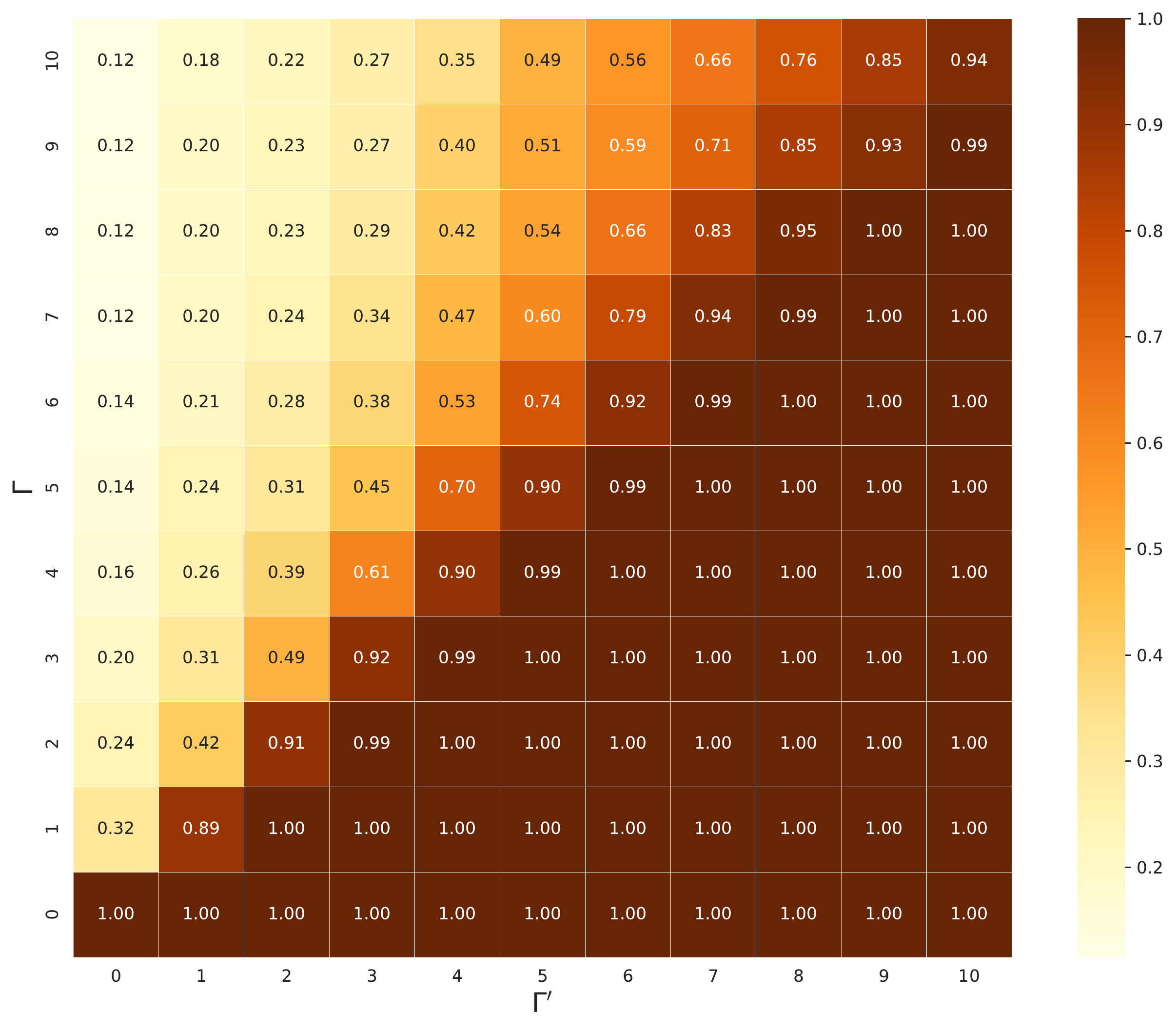}

\caption{Median runtime of the iterative method on shortest path instances from the Chicago testset with various $\Gamma$ and $\Gamma'$ (left) and the percentage of instances with objective value $0$ (right).\label{Fig::SPN50}}
\end{figure}

Furthermore, for $\Gamma=8$ we examine relative differences of the optimal balanced regret solutions with $\Gamma'\in \{1,\ldots,8\}$ with respect to the evaluation criteria. The results are shown in Table \ref{Tab::QualitySP}. Similar to the other experiments, the trend of \RR{$\Gamma'$} becoming a \WCI optimal solution for increasing $\Gamma'$ can be observed. Figure \ref{Fig::ScatterSP} shows the trade-off between \WCI and \RegretGamma. 
\begin{table}[h!]
\centering
\footnotesize
\caption{Average relative difference of the different optimization approaches on  shortest path instances compared to the optimal solutions of \BC \WCI, \WCGamma, \RegretI and \RegretGamma.\label{Tab::QualitySP}}

\begin{tabular}{llllll}
          & \BC    & \WCI & \WCGamma & \RegretI & \RegretGamma \\
          \midrule
\BC &0.000 & 0.024 & 0.011 & 0.101 & 0.064 \\
\WCI &0.035 & 0.000 & 0.009 & 0.036 & 0.095 \\
\WCGamma &0.019 & 0.011 & 0.000 & 0.060 & 0.030 \\
\RegretI &0.027 & 0.006 & 0.011 & 0.000 & 0.047 \\
\RegretGamma&0.014 & 0.012 & 0.003 & 0.030 & 0.000 \\
\RR{1}&0.016 & 0.010 & 0.002 & 0.031 & 0.006 \\
\RR{2}&0.018 & 0.009 & 0.002 & 0.037 & 0.015 \\
\RR{3}&0.020 & 0.008 & 0.001 & 0.036 & 0.019 \\
\RR{4}&0.020 & 0.007 & 0.001 & 0.036 & 0.021 \\
\RR{5}&0.022 & 0.006 & 0.001 & 0.042 & 0.029 \\
\RR{6}&0.024 & 0.005 & 0.002 & 0.044 & 0.037 \\
\RR{7}&0.026 & 0.004 & 0.002 & 0.044 & 0.040 \\
\RR{8}&0.027 & 0.004 & 0.002 & 0.044 & 0.042
\end{tabular}
\end{table}

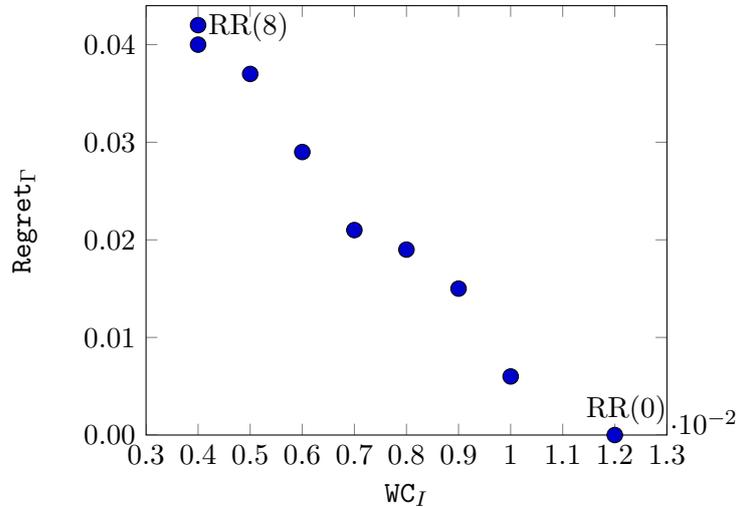
\begin{figure}[h!]
\centering
\begin{tikzpicture}

\begin{axis}[
  ylabel = {\RegretGamma},
    ylabel shift = -3 pt,
    xlabel = \WCI,
    ymin=0,
    ymax=0.044,
    xmax=0.013,
    xmin=0.003,
        enlargelimits=false,
        ylabel shift = 10 pt,
            xtick distance={0.001},
    ytick distance={0.01},
        y tick label style={
        /pgf/number format/.cd,
            fixed,
            fixed zerofill,
            precision=2,
        /tikz/.cd,
    },   
       scaled y ticks=false,
     every x tick scale label/.style={
    at={(1,0)},xshift=1pt,anchor=south west,inner sep=0pt
}
]
\addplot+[
    only marks,
    color=black,
    scatter,
    mark=*,
    mark size=2.9pt, 
    point meta=explicit symbolic,
    visualization depends on=\thisrow{alignment} \as \alignment,
    nodes near coords,
     every node near coord/.style={anchor=\alignment}]
table[meta=pin]
{shortestpathScatter.dat};
\end{axis}
\end{tikzpicture}
\caption{Evaluation of the \RR{$\Gamma'$} shortest path solutions with respect to \WCI and  \RegretGamma. 
 \label{Fig::ScatterSP}}
\end{figure}

\section{Conclusions}
\label{sec:conclusions}

Decision making under uncertainty is ubiquitous. While there is no ``perfect'' decision criterion that fulfills a complete set of reasonable axioms, several such criteria have received particular attention in the research literature, including the min-max and the min-max regret approach. In the latter, we compare our decision against an omniscient adversary that already has full knowledge of what the future will bring.

In this paper we introduced a new decision making approach for budgeted uncertainty sets, called balanced regret. In this setting, we ``level the playing field'' between decision maker and adversary by considering the adversary solution as being affected by uncertainty as well. Formally, this results in an additional optimization stage, requiring us to solve min-max-min problems. We proposed general-purpose solution methods and noticed that for sufficiently large uncertainty for the adversary, the problem becomes equivalent to a simple worst-case problem. We then considered the multi-representative selection problem in more detail, showing that it is NP-hard. We derived a compact problem formulation and a dominance criterion which allows us to solve special cases in polynomial time. Furthermore, we show that the classic regret case with budgeted uncertainty can be solved in polynomial time.

In computational experiments using three types of combinatorial optimization problems under randomly generated and real-world data, we analyzed our approach in more detail. Comparing solution methods, we showed that the compact formulation is the strongest approach in case of the selection problem, while the iterative method performs best for other problems, where no such formulation is available. We compare the balanced regret solution to solutions found by other decision making criteria and found that it provides a useful trade-off between the worst-case solution (with respect to interval uncertainty) and the min-max regret solution (with respect to budgeted uncertainty), thus providing more practical choices for a decision maker.

In further research we will analyze the approximability of this problem setting and derive heuristic solution methods for larger instances.

\newcommand{\etalchar}[1]{$^{#1}$}

\end{document}